\documentclass[11pt]{article}
\usepackage{pifont,amsmath,amssymb}
\usepackage{authblk}
\usepackage[english]{babel}
\usepackage[latin1]{inputenc}
\usepackage{times}
\usepackage{hyperref}
\usepackage{multimedia}
\usepackage[T1]{fontenc}
\usepackage{pifont}
\usepackage{mathrsfs}
\usepackage{euscript}
\usepackage{ifsym}
\usepackage{enumerate}
\usepackage{amsthm}
\allowdisplaybreaks

\newcommand{\N}{\ensuremath{\mathbb{N}}}

\newcommand{\C}{\ensuremath{\mathbb{C}}}

\newcommand{\dx}{\mathrm{d}}
\newcommand{\e}{\mathrm{e}}

\newcommand{\R}{\mathbb{R}}

\newtheorem{theorem}{Theorem}
\newtheorem{corollary}[theorem]{Corollary}
\newtheorem{lemma}[theorem]{Lemma}
\newtheorem{proposition}[theorem]{Proposition}
\newtheorem{definition}[theorem]{Definition}
\newtheorem{example}[theorem]{Example}
\newtheorem{algorithm}[theorem]{Algorithm}
\newtheorem{remark}[theorem]{Remark}

\makeatletter
\renewcommand*\env@cases[1][1.2]{%
  \let\@ifnextchar\new@ifnextchar
  \left\lbrace
  \def\arraystretch{#1}%
  \array{@{}l@{\quad}l@{}}%
}
\makeatother

\newenvironment{Lemma}{\goodbreak\begin{lemma}}{\end{lemma}}
\newenvironment{Theorem}{\goodbreak \begin{theorem}}{\end{theorem}}

\providecommand{\keywords}[1]{\textbf{\textit{Keywords: }} #1}

\author{Frank Filbir,  Kristof Schr\"{o}der\footnote{Corresponding author: kristof.schroeder@helmholtz-muenchen.de}}
\affil{\small Department of Mathematics, Technische Universit\"at M\"unchen (Germany),
\\
Scientific Computing Research Unit, Helmholtz Center Munich (Germany)}
\title{Exact Recovery of Discrete Measures from Wigner D-Moments}
\date{}
\begin{document}
\maketitle
\begin{abstract}
In this paper, we show the possibility of recovering a sum of Dirac measures on the rotation group $SO(3)$ from its low degree moments with respect to \emph{Wigner D-functions} only. 

The main Theorem of the paper, Theorem \ref{exact:recmain}, states, that exact recovery from moments up to degree $N$ is possible, if the support set of the measure obeys a separation distance of $\frac{36}{N+1}$. In this case, the sought measure is the unique solution of a \emph{total variation} minimization. The proof of the uniqueness of the solution is  in the spirit of the work of Cand\'es and Fernandez-Granda \cite{Candes:2014} and requires \emph{localization estimates} for interpolation kernels and corresponding \emph{derivatives} on the rotation group $SO(3)$ with \emph{explicit} constants.  
\end{abstract}

\section{Introduction}
The group of all rotation matrices in dimension three plays a crucial role in various applications ranging from crystallographic texture analysis, see \cite{Bunge:1982}, \cite{Hielscher:2008},\cite{Hielscher:2007}, \cite{Schaeben:2003}, over the calculation of magnetic resonance spectra \cite{Stevensson:2011} to applications in biology such as protein-protein docking, see \cite{Candas:2005}, \cite{Bajaj:2013}, \cite{Kovacs:2003}. For a good overview regarding applications see also \cite{EngAppl:2000}. 

Signals or functions on the rotation group $SO(3)$ are often analysed with respect to a harmonic basis arising from representation theory of the group, the so called \emph{Wigner D-functions}. Since these functions are also eigenfunctions of the Laplace operator on the manifold $SO(3)$, they can be regarded as a natural analog to Fourier series in the case of the torus group. Also fast algorithms in the spirit of the FFT for analysing and summation with respect to this system of functions have been developed, see \cite{Kostelec:2008}, \cite{Potts:2009}, \cite{Hielscher:2010} and \cite{Keiner:2012}. 

In this paper, we consider the problem of recovering a spatially highly resolved signal, modelled as a sum of point measures on $SO(3)$ denoted by $\mu^{\star}$, from its low order moments with respect to the Wigner D-functions only. This sort of problem has been treated in different geometric settings and with respect to different
systems of functions. 

To give a short overview, we first mention that there is a long list of works on the recovery from \emph{classical Fourier measurements}, i.e. Fourier coefficients for measures on the torus group $\mathbb{T}$ or Fourier transform measurements on the real line. In Section $1.8$ of \cite{Candes:2014}, Cand\'es and Fernandez-Granda give an extensive overview about the work in this direction. Moreover, they show that the sought measure is the unique solution of a \emph{total variation} minimization problem as long as the support of the measure is sufficiently separated, i.e. obey a minimal separation condition of $\frac{2}{N}$, where $N$ denotes the order of available Fourier coefficients. The proportional factor of $2$ is coined \emph{super-resolution factor}.        
 
The corner stone is a measure theoretic proof that involves the construction of a so called \emph{dual certificate}, i.e. a trigonemetric polynomial of degree $N$ that interpolates a given sign sequence on the support of the sought measure and is strictly less than one in absolute value elsewhere. In addition, they formulated the recovery as a tractable optimization problem and studied in \cite{Candes:2013} the robustness of this procedure with respect to noise. Very recently their procedure has been generalized to \emph{Short-Time} Fourier measurements, see \cite{Boelcskei:2015}, and a connection to \emph{Beurling's} theory of minimal extrapolation was shown in \cite{Benedetto:2016}.

A different line of work regards different geometric settings and different moments like an interval or the whole line with polynomial or generalized moments, see \cite{Castro:2012}, \cite{Bendory:2014}, or semi-algebraic domains \cite{deCastro:2015}. In \cite{Bendory:2015} and \cite{Bendory:2015a} the possibility of recovery was shown for measures on the two-dimensional sphere from moments with respect to spherical harmonics.          

In this work, we show, that the recovery on the rotation group $SO(3)$ can be stated also as an \emph{TV-minimization} problem and the basic principle of the proof of uniqueness can be carried over from the trigonemetric case shown in \cite{Candes:2014}. Nevertheless, the actual construction of a \emph{dual certificate} on $SO(3)$ requires \emph{localization} estimates for \emph{interpolation kernels} and corresponding \emph{derivatives} with \emph{explicit} constants, which are of interest on its own. This is the main contribution of this work.   
 
The paper is organized in the following way. In Section \ref{sec:ex_rec}, we introduce the necessary analysis on the rotation group including \emph{Wigner D-functions}, state the problem of \emph{exact recovery} and point out the connection to the \emph{total variation minimization} problem and the existence of a \emph{dual certificate}. Section \ref{sec:localker} provides \emph{locality results} for kernels on the rotation group, that are the key ingredient for the choosen construction of the dual certificate. The actual construction is presented in Section \ref{sec:dualconstr}, where we show the explicit \emph{super-resolution factor}. A short conclusion and outlook is given in Section \ref{sec:conclu}.

\section{Rotation Group and Exact Recovery}\label{sec:ex_rec}
Here we give a short reminder on the analysis on the rotation group including \emph{Wigner D-functions} and state the problem of \emph{exact recovery}. Moreover we show the connection to the TV-minimization problem and to the
existence of a so called \emph{dual certificate}.

\subsection{Analysis on SO(3)}\label{subsec:rotana}
The rotation group $SO(3)$ is defined as the space of matrices
\begin{equation*}
SO(3):=\{x \in \R^{3 \times 3}: x^Tx=I, \operatorname{det}x=1\},
\end{equation*}
which is a group under the action of matrix multiplication. By \emph{Euler's Rotation Theorem}, there is for each $A \in SO(3)$ a unit vector $e \in \R^3$ and an angle $\omega \in [0,\pi]$, such that $A$ is a rotation with rotation axis $e$ and rotation angle $\omega$.
Using \emph{Rodrigues rotation formula} yields
\begin{equation*}
A=I\cos(\omega)+(1-\cos(\omega))ee^T+[e]\sin(\omega),
\end{equation*} 
where
\begin{equation*}
[e]=\begin{bmatrix}
0 & -e_3 & e_2\\
e_3 &0& -e_1\\
-e_2 & e_1 & 0
\end{bmatrix}.
\end{equation*}
This identification shows that $SO(3)$ is diffeomorphic to the real three-dimensional projective space and is therefore a connected compact Lie group.
A Metric on $SO(3)$, that is compatible with this topology and invariant with respect to the group action, is given by
\begin{equation*}
d(x,y)=\omega(y^{-1}x)=\arccos\left(\frac{\operatorname{tr}(y^{-1}x)-1}{2}\right),
\end{equation*}
which is equal to the rotation angle of the matrix $y^{-1}x$.  The corresponding Lie algebra of the Lie group $SO(3)$ is given by the skew symmetric matrices, i.e.
\begin{equation*}
\mathfrak{so}(3)=\{x \in \R^{3\times3}: x^T=-x\}.
\end{equation*}
The generators of the Lie-Algebra $\mathfrak{so}(3)$ are given by 
\begin{equation*}
\mathcal{L}_1=\begin{pmatrix}
0&0&0 \\ 0 &0 & -1\\ 0 & 1 &0
\end{pmatrix},\quad
\mathcal{L}_2=\begin{pmatrix}
0&0&1 \\ 0 &0 & 0\\ -1 & 0 &0
\end{pmatrix}, \quad
\mathcal{L}_3=\begin{pmatrix}
0&-1&0 \\ 1 &0 & 0\\ 0 & 0 &0
\end{pmatrix}.
\end{equation*}
The corresponding elements in $SO(3)$ are given for $t \in \R$ by
\begin{align*}
&\e^{t\mathcal{L}_1}=\begin{pmatrix}
1 & 0 & 0\\ 0 & \cos(t) & -\sin(t)\\ 0 & \sin(t) & \cos(t)
\end{pmatrix},
\e^{t\mathcal{L}_2}=\begin{pmatrix}
\cos(t) & 0 & \sin(t)\\ 0 & 1 & 0\\ -\sin(t) & 0 & \cos(t)
\end{pmatrix},\\
&\e^{t\mathcal{L}_3}=\begin{pmatrix}
\cos(t) & -\sin(t) & 0\\ \sin(t) & \cos(t) & 0 \\ 0 & 0 & 1
\end{pmatrix},
\end{align*}
where $\e^{A}=\sum_k\frac{A^k}{k!}$ denotes the matrix exponential. The matrix exponential is used to define differential operators along these directions, given by 
\begin{equation*}
X_if(x)=\lim_{t\rightarrow 0} t^{-1}(f(xe^{t\mathcal{L}_i})-f(x)),
\end{equation*}
for a differentiable function $f : SO(3) \rightarrow \C$. For a two times differentiable function $f$ the \emph{Hessian matrix} is given by 
\begin{equation*}\label{def:hessian}
Hf=
\begin{pmatrix}
X_1X_1f & X_1X_2f & X_1X_3f\\
X_2X_1f & X_2X_2f & X_2X_3f\\
X_3X_1f & X_3X_2f & X_3X_3f
\end{pmatrix}
-\frac{1}{2}\begin{pmatrix}
0 & X_3f & -X_2f\\
-X_3f & 0 & X_1f\\
X_2f & -X_1f & 0
\end{pmatrix}.
\end{equation*}

Since $SO(3)$ is a compact group, there is a regular Borel measure $\lambda$, that is invariant under the group action, i.e. $\lambda(xB)=\lambda(B)=\lambda(Bx)$ for all Borel sets $B$. This measure can be normalized such that
\begin{equation*}
\int_{SO(3)}\dx \lambda(x)=1.
\end{equation*}  

Using an Euler angle parametrization, i.e. each element $x \in SO(3)$ is represented by
\begin{equation*}
x=R_{Z}(\alpha)R_{X}(\beta)R_{Z}(\gamma),
\end{equation*}
with $(\alpha,\beta,\gamma) \in [0,2\pi) \times [0,\pi] \times [0,2\pi)$ and 
\begin{equation*}
R_Z(t)=\begin{pmatrix}
\cos(t) & -\sin(t) & 0\\ \sin(t) & \cos(t) & 0\\ 0 & 0& 1
\end{pmatrix},
\qquad
R_X(t)=\begin{pmatrix}
1 & 0 & 0\\ 0 & \cos(t) & -\sin(t)\\ 0 & \sin(t)& \cos(t)
\end{pmatrix},
\end{equation*}
we can write down the integral for each measurable function $f: SO(3) \rightarrow \C$ explicitly as
\begin{equation*}
\int_{SO(3)}f(x)\dx \lambda(x)=\frac{1}{8\pi^2}\int_{0}^{2\pi}\int_0^{\pi}\int_0^{2\pi}f(x(\alpha,\beta,\gamma))\sin(\beta)\dx \alpha \dx \beta \dx \gamma.
\end{equation*}
For functions that only depend on the rotation angle, i.e. $f(x)=\tilde{f}(\omega(x))$ the integral reduces to
\begin{equation*}
\int_{SO(3)}f(x)\dx \lambda(x)=\frac{2}{\pi}\int_0^{\pi}\tilde{f}(t) \sin^2\left(\frac{t}{2}\right)\dx t.
\end{equation*}
The space $L^2(SO(3))$ of square-integrable functions with respect to $\lambda$ is defined in the usual way. The Peter-Weyl Theorem now states that the right regular representation of $SO(3)$ splits up into an orthogonal direct sum of irreducible finite-dimensional representations and the matrix coefficients of these irreducible representation form an orthogonal basis for $L^2(SO(3))$. The dimensions of the irreducible representations are given by $2l+1$, $l \in \N$ and the matrix coefficients $D^l_{k,m}$, $-l \leq k,m \leq l$ are often called \emph{Wigner D-functions}. We have that
\begin{equation*}
\{\sqrt{2l+1}D^l_{k,m},-l\leq k,m\leq l, l \in \N\}
\end{equation*}   
form an orthonormal basis of $L^2(SO(3))$. The value $l\in \N$ will be called \emph{degree}. 
In the Euler angle parametrization the Wigner $D$-functions are given for $l \in \N$ and $-l \leq k,m \leq l$ by
\begin{equation*}
D^l_{k,m}(\alpha,\beta,\gamma)=\e^{-ik\alpha}P^l_{k,m}(\cos(\beta))\e^{-im\gamma},
\end{equation*}
where $P^l_{k,m}$ is given by
\begin{equation*}
P^l_{k,m}(t)=C_{l,k,m}(1-t)^{-(m-k)/2}(1+t)^{-(m+k)/2}\frac{d^{l-m}}{d t^{l-m}}\left((1-t)^{l-k}(1+t)^{l+k}\right),
\end{equation*}
with $C_{l,k,m}= \frac{(-1)^{l-k}i^{m-k}}{2^l(l-k)!}\sqrt{\frac{(l-k)!(l+m)!}{(l+k)!(l-m)!}}$. 
 The space of all finite linear combinations of Wigner D-functions with degree less or equal to $N$ will be denoted as
\begin{equation*}
\Pi_N:=\operatorname{span}\{D_{k,m}^l: \; -l \leq k,m\leq l,\; l \leq N \}
\end{equation*}
and will also be called \emph{generalized polynomials} of degree $N$.

\subsection{Exact Recovery from Wigner D-moments}
In the following, we will introduce the problem of exact recovery of dirac measures from its moments with respect to these Wigner D-functions up to a degree $N$, which means we can access the moments of $D^l_{k,m}$ for $-l \leq k,m \leq l$ only for $l \leq N$.

Consider a dirac measure of the form
\begin{equation}\label{point:meas}
\mu^{*}=\sum_{i=1}^Mc_i\delta_{x_i},
\end{equation} 
where $M \in \N$, $c_i\in \R$ are real valued coefficients and $\delta_{x_i}$ are the point measures centred at pairwise distinct $x_i \in SO(3)$. All parameters $M, c_i, x_i$ are unknown and we can only access
\begin{equation*}
\langle \mu^{*}, D^l_{k,m} \rangle:=\int_{SO(3)}D^l_{k,m}(x) \dx \mu(x)=\sum_{i=1}^Mc_iD^l_{k,m}(x_i),  
\end{equation*}
for $-l \leq k,m \leq l, l \leq N$. 
The main Theorem of this paper states, that ,under a suitable condition on the \emph{separation distance}  
\begin{equation*}
\rho(\mathcal{C}):=\min_{x_i\neq x_j}\omega(x_j^{-1}x_i)
\end{equation*}
of the support points $\mathcal{C}=\{x_1,\dots,x_M\}$, $\mu^{*}$ is the unique measure of \emph{minimal total variation} that obeys the prescribed moments.   
For a signed Borel measure $\mu \in \mathcal{M}(SO(3))$ we define its \emph{total variation} by
\begin{equation*}
\|\mu\|_{TV}=|\mu|(SO(3))=\sup\sum_{j}|\mu(B_j)|,
\end{equation*} 
where the supremum is taken over all Partitions $B_j$ of $SO(3)$. 

\begin{Theorem}\label{exact:recmain}
Suppose the support points $\mathcal{C}=\{x_1,\dots, x_M\}$ of the measure $\mu^{*}$, given in \eqref{point:meas}, obey a separation distance of $\rho(\mathcal{C})\geq \frac{36}{N+1}$ for $N\geq 20$. Then $\mu^{*}$ is the unique solution of the minimization problem
\begin{align}\label{opti:prob}
\min_{\mu \in \mathcal{M}(SO(3))}\|\mu\|_{TV},\quad \text{subject to} \quad &\langle\mu, D^l_{k,m} \rangle=\langle\mu^{*}, D^l_{k,m} \rangle,\\
&\text{for } -l\leq k,m \leq l \text{ and } l \leq N.\nonumber
\end{align}
\end{Theorem}
The entry point for the proof is the connection of the solution of the optimization problem \eqref{opti:prob} to the existence of a so called \emph{dual certificate}.   
This connection has been exploited in various settings, see e.g. \cite{Castro:2012}, \cite{Candes:2014}, \cite{Bendory:2014} and \cite{Bendory:2015}.
We will state this connection for the rotation group in the following Theorem.
\begin{Theorem}\label{exact:recdual}
Let $\mu^{*}=\sum_{i=1}^Mc_i\delta_{x_i}$ with $c_i \in \R$ and $\mathcal{C}=\{x_1,\dots,x_M\} \subset SO(3)$. If for all ordered sets $(u_i)_{i=1}^m \in \{-1,1\}^m$ there is a function $q \in \Pi_N$, called \emph{dual certificate}, such that
\begin{align}\label{interpol:cond1}
q(x_i)&=u_i, \quad \text{for } x_i \in \mathcal{C},\\
|q(x)| &<1,\quad \text{for } x \in SO(3)\setminus \mathcal{C},\nonumber
\end{align}
then $\mu^{*}$ is the unique solution of the optimization problem \eqref{opti:prob}.
\end{Theorem}
The proof is purely measure theoretic and can be easily transferred from those proofs found in the references cited above. We will therefore omit it here.

To explicitly construct a dual certificate, we will borrow ideas from \cite{Candes:2014}, where the construction was done for \emph{trigonometric polynomials} and was later adapted to the case of \emph{algebraic polynomials} in \cite{Bendory:2014} and \emph{spherical harmonics} in \cite{Bendory:2015}. 

In order to satisfy the conditions $\eqref{interpol:cond1}$, one formulates the \emph{Hermite-type} interpolation problem
\begin{align}\label{interpol:cond2}
q(x_i)&=u_i,\\
X_1q(x_i)&=X_2q(x_i)=X_3q(x_i)=0,\nonumber
\end{align}
for $x_i \in \mathcal{C}$, where $X_k$ are the differential operators defined in Section \ref{subsec:rotana}. This means, besides the interpolation itself, we ask for local extrema at the interpolation points. One then seeks a solution $q$ to this interpolation problem in the space $\Pi_N$, that satisfies, due to the local extrema conditions, $|q(x)| < 1$ for $x \in SO(3)\setminus \mathcal{C}$. The constructed interpolant is of the form
\begin{align*}
q(x)=\sum_{i=1}^M\alpha_{i,0}\sigma_N(x,x_i)+ \alpha_{i,1}X_1^y\sigma_N(x,x_i)+  \alpha_{j,2}X_2^y\sigma_N(x,x_i)+  \alpha_{j,3}X_3^y\sigma_N(x,x_i),
\end{align*}
where $\sigma_N$ is an \emph{interpolation kernel} of the form 
\begin{equation*}
\sigma_N(x,y)=\sum_{l=0}^{N}h_{N}(l)\sum_{-l \leq k,m \leq l}D^l_{k,m}(x)\overline{D^l_{k,m}(y)},
\end{equation*}
with positive weights $h_N(l) > 0$. Observe, that the expressions $\sigma_N(x,x_i)$ and $X_j^y\sigma_N(x,x_i)$, where the superscript indicates the action of the differential operator on the second variable, are by construction generalized polynomials of degree $N$ in the first variable, which means $q \in \Pi_N$. Applying the interpolation conditions \eqref{interpol:cond2} will lead to the linear system of equations

\begin{equation}\label{interpol:mat}
K\alpha:=
\begin{pmatrix}
\sigma_N & X^x_1\sigma_N & X^x_2\sigma_N & X^x_3\sigma_N\\
X^y_1\sigma_N & X^x_1X^y_1\sigma_N & X^x_2X^y_1\sigma_N & X^x_3X^y_1\sigma_N \\
X^y_2\sigma_N & X^x_1X^y_2\sigma_N & X^x_2X^y_2\sigma_N & X^x_3X^y_2\sigma_N \\
X^y_3\sigma_N & X^x_1X^y_3\sigma_N & X^x_2X^y_3\sigma_N & X^x_3X^y_3\sigma_N
\end{pmatrix}
\begin{pmatrix}
\alpha_0\\
\alpha_1\\
\alpha_2\\
\alpha_3
\end{pmatrix}
=
\begin{pmatrix}
u\\
0\\
0\\
0\\
\end{pmatrix},
\end{equation}
where the entries in the matrix corresponds to blocks of the form $\sigma_N=(\sigma_N(x_i,x_j))_{i,j=1}^M$ and in the same way for the derivatives. The blocks in the vectors are given by $\alpha_k=(\alpha_{k,j})_{j=1}^M$ for $k=0,1,2,3$ and $u=(u_j)_{j=1}^M$. To find the coefficients, we have to show the invertibility of the matrix $K$. Due to the block structure of $K$ this is done using an \emph{iterative block inversion}, explained in Section \ref{sec:dualconstr}, and the fact 
that a matrix $A$ is invertible if
\begin{equation*}
\|I-A\|_{\infty} < 1,
\end{equation*} 
where $\|A\|_{\infty}=\max_{i}\sum_{j}|a_{i,j}|$. In this case the norm of the inverse is bounded by
\begin{equation*}
\|A^{-1}\|_{\infty}\leq \frac{1}{1-\|I-A\|_{\infty}}.
\end{equation*}
Thus, to show the invertibility of the matrix $K$, we have to employ \emph{localization estimates} for the entries of the matrix $K$, which means we have to bound the expressions $|\sigma_N(x_i,x_j)|, |X_k^y\sigma_N(x_i,x_j)|$ and $|X_n^xX^y_k\sigma_N(x_i,x_j)|$. The values of these expressions should decrease, if the distance of $\omega(x_j^{-1}x_i)$ becomes bigger. We are locking for estimates of the form
\begin{equation*}
|\sigma_N(x_i,x_j)| \leq \frac{c}{((N+1)\omega(y^{-1}x))^s}  
\end{equation*}
for some constants $s$ and $c$ only depending on the weights $h_N$ and similar estimates for the derivatives. Using these estimates we find explicit bounds on the supremum norm of the coefficients. Once we have found the coefficients, we have to show the condition $|q(x)|< 1$, where $x$ is not an interpolation point. 
This includes convexity arguments for the interpolant $q$, which means we have to deal with the entries of the \emph{Hessian matrix} of $q$, where third mixed derivatives appear. Therefore we also need localization estimates for third derivatives.

The key ingredients for the construction of the interpolant $q$ are \emph{localization estimates} for the interpolation kernel $\sigma_N$ and its various derivatives. Moreover, we need \emph{explicit constants} in these estimates to show the claimed properties of the interpolant. This is the topic of the next section.

\section{Localized Kernels}\label{sec:localker}
The localization properties of the interpolation kernel and its derivatives for special choices of weights can be derived from corresponding localization principles for \emph{trigonometric polynomials}.
For the kernel itself, this was shown in \cite{Kunis:2008}. We define the kernel 
\begin{equation*}
\sigma_N(x,y)=\sum_{l=0}^{N}h_{N}(l)\sum_{-l \leq k,m \leq l}D^l_{k,m}(x)\overline{D^l_{k,m}(y)},
\end{equation*}
with positive weights $h_N(l)>0$. By the addition formula of Wigner $D$-functions this can also be written as
\begin{equation*}
\sigma_N(x,y)=\sum_{l=0}^{N}h_{N}(l)U_{2l}\left(\cos\left(\frac{\omega(y^{-1}x)}{2}\right)\right)=
\sum_{l=0}^{N}h_{N}(l)\sum_{k=-l}^{l}\e^{ik\omega(y^{-1}x)},
\end{equation*}
where $U_{2l}$ is the Chebyshev polynomial of the second kind of order $2l$. Now we define the filter coefficients $h_{N}$ by
\begin{equation*}
h_{N}(l)=\frac{1}{\|g\|_{1,N}}\begin{cases}
g\left(\frac{l}{2(N+1)}\right)-g\left(\frac{l+1}{2(N+1)}\right),& 0\leq l < N,\\
g\left(\frac{N}{2(N+1)}\right), & l=N,
\end{cases}
\end{equation*}
where $g :\R \rightarrow \R_+$ is a symmetric positive function with $\operatorname{supp}(g) \subseteq [-\frac{1}{2},\frac{1}{2}]$, that is decreasing for positive values. Its discrete coefficient norm is given
by\begin{equation*}
\|g\|_{1,N}:=\sum_{l=-N}^N g\left(\frac{l}{2(N+1)}\right).
\end{equation*}
Plugging this in, leads to
\begin{equation*}
\sigma_N(x,y)=\tilde{\sigma}_N(\omega(y^{-1}x)):=\frac{1}{\|g\|_{1,N}}\sum_{k=-N}^Ng\left(\frac{k}{2(N+1)}\right)\e^{ik\omega(y^{-1}x)}.
\end{equation*}
This shows, that the kernel $\sigma_N$ is a zonal kernel, i.e. its value only depends on the distance of $x$ and $y$, and localization estimates are derived from localization principles for the trigonometric polynomial $\tilde{\sigma}_N$. In appendix \ref{app:loctrig} we choose a specific filter function, given by a \emph{B-spline} of order $s$, to derive estimates of the form
\begin{equation*}
|\tilde{\sigma}_N^{(l)}(t)| \leq \frac{c_{l,s}}{(N+1)^{s-l}|t|^s}, \quad l=0,\dots,3,
\end{equation*}
with explicit constants $c_{l,s}$.

Since we have the equality
\begin{equation*}
\sigma_N(x,y)=\tilde{\sigma}_N(\omega(y^{-1}x)),
\end{equation*}
we have immediately
\begin{equation}\label{local:kernel0}
|\sigma_N(x,y)| \leq \frac{c_{0,s}}{((N+1)\omega(y^{-1}x))^s},
\end{equation}
which shows the localization of the kernel $\sigma_N$. The derivative kernels $X_n^y\sigma_N, X_i^xX_n^y\sigma_N$ and $X_j^xX_i^xX_n^y\sigma_N$ are no longer zonal functions. 
Nevertheless, the obey analog localization estimates with the same constants as in the trigonometric case. Thereby Theorem \ref{local:kernelall} provides estimates for the entries of the interpolation matrix in \eqref{interpol:mat},
whereas Lemma \ref{local:kernelall2} and \ref{lip:est} give bounds for the entries of the Hessian.
\begin{Theorem}\label{local:kernelall}
We have for $s \in 2\N$, $s \geq 6$, $N\geq 2s$, $\omega(y^{-1}x)\geq \frac{\pi}{2(N+1)}$
\begin{align*}
|X_n^y\sigma_N(x,y)|&\leq \frac{c_{1,s}}{(N+1)^{s-1}\omega(y^{-1}x)^s},\\
|X_i^xX_n^y\sigma_N(x,y)|&\leq \frac{c_{2,s}}{(N+1)^{s-2}\omega(y^{-1}x)^s},
\end{align*}
and $c_{l,s}$  are the constants of Theorem \ref{trig:est}.
\end{Theorem}
\begin{proof}
We calculate the derivative kernel $X_1^y\sigma_N$. For $\omega(y^{-1}x) \notin \{0,\pi\}$, we have
\begin{align*}
& \frac{\sigma_N(x,y \e^{t\mathcal{L}_1})-\sigma_N(x,y)}{t}\\
&=\frac{\tilde{\sigma}_N(\omega(\e^{-t\mathcal{L}_1}y^{-1}x))-\tilde{\sigma}_N(\omega(y^{-1}x))}{\omega(\e^{-t\mathcal{L}_1}y^{-1}x)-\omega(y^{-1}x)}\frac{\omega(\e^{-t\mathcal{L}_1}y^{-1}x)-\omega(y^{-1}x)}{\operatorname{tr}(\e^{-t\mathcal{L}_1}y^{-1}x)-\operatorname{tr}(y^{-1}x)}\\
&\quad \frac{\operatorname{tr}(\e^{-t\mathcal{L}_1}y^{-1}x)-\operatorname{tr}(y^{-1}x)}{t}.
\end{align*} 
The limits are given by 
\begin{equation*}
\lim_{t \rightarrow 0} \frac{\omega(\e^{-t\mathcal{L}_1}y^{-1}x)-\omega(y^{-1}x)}{\operatorname{tr}(\e^{-t\mathcal{L}_1}y^{-1}x)-\operatorname{tr}(y^{-1}x)}=\frac{1}{-2\sqrt{1-(\frac{\operatorname{tr}(y^{-1}x)-1}{2})^2}}
\end{equation*}
and
\begin{equation*}
\lim_{t \rightarrow 0} \frac{\operatorname{tr}(\e^{-t\mathcal{L}_1}y^{-1}x)-\operatorname{tr}(y^{-1}x)}{t}=((y^{-1}x)_{32}-(y^{-1}x)_{23}).
\end{equation*}
Therefore
\begin{align*}
X_1^{y}\sigma_N(x,y)&=\tilde{\sigma}_N'(\omega(y^{-1}x))\frac{((y^{-1}x)_{32}-(y^{-1}x)_{23})}{-2\sqrt{1-(\frac{\operatorname{tr}(y^{-1}x)-1}{2})^2}}\\
&=\tilde{\sigma}_N'(\omega(y^{-1}x))\frac{((y^{-1}x)_{32}-(y^{-1}x)_{23})}{-2\sin(\omega(y^{-1}x))}
&=-\tilde{\sigma}_N'(\omega(y^{-1}x))e_1,
\end{align*}
where $e_1=e_1(y^{-1}x)$ is the first component of the unit vector describing the rotation axis of $y^{-1}x$.
In the same way one can calculate
\begin{align*}
X_2^{y}\sigma_N(x,y)&=\tilde{\sigma}_N'(\omega(y^{-1}x))\frac{((y^{-1}x)_{31}-(y^{-1}x)_{13})}{2\sin(\omega(y^{-1}x))}=-\tilde{\sigma}_N'(\omega(y^{-1}x))e_2,\\
X_3^{y}\sigma_N(x,y)&=\tilde{\sigma}_N'(\omega(y^{-1}x))\frac{((y^{-1}x)_{12}-(y^{-1}x)_{21})}{2\sin(\omega(y^{-1}x))}=-\tilde{\sigma}_N'(\omega(y^{-1}x))e_3.
\end{align*}
Also observe that we have
\begin{equation*}
X_n^x\sigma_N(x,y)=-X_n^y\sigma_N(x,y).
\end{equation*}
These expressions are valid for all $x,y \in SO(3)$ with $\operatorname{tr}(y^{-1}x) \notin \{\-1,3\}$. 
We know that $X_{n}^y\sigma_N(x,y)$ is always a finite sum of products of Wigner D-functions, since each operator $X_i$ maps a Wigner D-function to sums of Wigner D-functions, see e.g. \cite{EngAppl:2000}.
Thus, we know for a fixed $x \in SO(3)$ that $X_{n}^y\sigma_N(x,y)$ exists for all $y \in SO(3)$ and is continuous,
which means that by limit considerations the expressions above are also valid if $\omega(y^{-1}x)= \pi$. In the case $y=x$, we have by limit considerations $X_n^y\sigma_N(x,x) =\tilde{\sigma}_N'(0)=0$. 
This leads to
\begin{equation*}
|X_n^y\sigma_N(x,y)|\leq |\tilde{\sigma}_N'(\omega(y^{-1}x))|\leq\frac{c_{1,s}}{(N+1)^{s-1}\omega(y^{-1}x)^s},
\end{equation*}   
which gives the estimate for the first type of kernel.

For the estimation of the second kind of kernel we use the product rule and the calculations above to show
\begin{equation*}
X_i^xX_n^y\sigma_N(x,y)=- X_i^xe_n(x,y)\tilde{\sigma}_N^{'}(\omega(y^{-1}x))- e_n(y^{-1}x)X_i^x(\tilde{\sigma}_N'(\omega(y^{-1}x))).
\end{equation*}
In the same way as before we can show
\begin{equation*}
X_i^x(\tilde{\sigma}_N'(\omega(y^{-1}x)))= \tilde{\sigma}_N^{''}(\omega(y^{-1}x))e_i(y^{-1}x),
\end{equation*}
and thus
\begin{equation*}
X_i^xX_n^y\sigma_N(x,y)=- X_i^xe_n(x,y) \tilde{\sigma}_N'(\omega(y^{-1}x))- e_n(y^{-1}x)\tilde{\sigma}_N^{''}(\omega(y^{-1}x))e_i(y^{-1}x).
\end{equation*}

Thus, the only part we have to calculate is $X_i^xe_n(x,y)$. Again, we restrict ourself firstly to $\omega(y^{-1}x) \notin \{0,\pi\}$ and extend afterwards by continuity. We concentrate on the example $n=1,i=3$. We have
\begin{align*}
& e_1(y^{-1}x\e^{t \mathcal{L}_3})-e_1(y^{-1}x)\\
&= \left[\frac{(y^{-1}x\e^{t\mathcal{L}_3})_{32}-(y^{-1}x\e^{t\mathcal{L}_3})_{23}}{2\sin(\omega(y^{-1}x\e^{t\mathcal{L}_3}))} - \frac{(y^{-1}x)_{32}-(y^{-1}x)_{23}}{2\sin(\omega(y^{-1}x))} \right],\\
&=\frac{1}{2\sin(\omega(y^{-1}x\e^{t \mathcal{L}_3}))}\bigg[(y^{-1}x)_{32}\left(\cos(t)-\frac{\sin(\omega(y^{-1}x\e^{t\mathcal{L}_3}))}{\sin(\omega(y^{-1}x))}\right)+\dots \\
&\dots (y^{-1}x)_{23}\left(\frac{\sin(\omega(y^{-1}x\e^{t\mathcal{L}_3}))}{\sin(\omega(y^{-1}x))}-1\right)-(y^{-1}x)_{31}\sin(t)\bigg]. 
\end{align*}

%
Using the rule of L'H$\hat{\text{o}}$pital we have
\begin{equation*}
\lim_{t \rightarrow 0}\frac{\cos(t)-\frac{\sin(\omega(y^{-1}x\e^{t\mathcal{L}_3}))}{\sin(\omega(y^{-1}x))}}{t}=-e_3(y^{-1}x)\bigg( \frac{\cos(\omega(y^{-1}x))}{\sin(\omega(y^{-1}x))}\bigg),
\end{equation*}
where $e_3(y^{-1}x)$ denotes the third component of the unit vector representing the rotation axis of $y^{-1}x$. 
In the same way we get
\begin{equation*}
\lim_{t \rightarrow 0}\frac{\frac{\sin(\omega(y^{-1}x\e^{t\mathcal{L}_3}))}{\sin(\omega(y^{-1}x))}-1}{t}=e_3(y^{-1}x)\bigg( \frac{\cos(\omega(y^{-1}x))}{\sin(\omega(y^{-1}x))}\bigg).
\end{equation*}

Combining all this, we end up with
\begin{align*}
X_3^xe_1(x,y)&=\lim_{t \rightarrow 0} t^{-1} (e_1(y^{-1}x\e^{t \mathcal{L}_3})-e_1(y^{-1}x));\\
&=\frac{1}{2\sin(\omega(y^{-1}x))}\left[ ((y^{-1}x)_{23}-(y^{-1}x)_{32})e_3(y^{-1}x)\bigg( \frac{\cos(\omega(y^{-1}x))}{\sin(\omega(y^{-1}x))}\bigg)-(y^{-1}x)_{31}\right]\\
&=-e_1(y^{-1}x)e_3(y^{-1}x)\bigg( \frac{\cos(\omega(y^{-1}x))}{\sin(\omega(y^{-1}x))}\bigg)-\frac{(y^{-1}x)_{31}}{2\sin(\omega(y^{-1}x))}.
\end{align*} 
Now we again use the \emph{Rodrigues formula} for $(y^{-1}x)_{31}=(1-\cos(\omega))e_1e_3-\sin(\omega)e_2$ and get
\begin{equation*}
X_3^xe_1(x,y)=-\frac{e_1(y^{-1}x)e_3(y^{-1}x)(1+\cos(\omega(y^{-1}x)))}{2\sin(\omega(y^{-1}x))}+\frac{e_2(y^{-1}x)}{2}.
\end{equation*}
Similarly we can calculate
\begin{equation*}
X_2^xe_1(x,y)=-\frac{e_1(y^{-1}x)e_2(y^{-1}x)(1+\cos(\omega(y^{-1}x)))}{2\sin(\omega(y^{-1}x))}-\frac{e_3(y^{-1}x)}{2}
\end{equation*}
and
\begin{equation*}
X_1^xe_1(x,y)=\frac{1+\cos(\omega(y^{-1}x))}{2\sin(\omega(y^{-1}x))}(1-e_1(y^{-1}x)^2).
\end{equation*}
For the other components of the rotation axis the differentials are computed in the same way and are given by
\begin{align}\label{deriv:rotax}
X_1^xe_2(x,y)&=-\frac{e_1(y^{-1}x)e_2(y^{-1}x)(1+\cos(\omega(y^{-1}x)))}{2\sin(\omega(y^{-1}x))}+\frac{e_3(y^{-1}x)}{2},\nonumber\\
X_2^xe_2(x,y)&=\frac{1+\cos(\omega(y^{-1}x))}{2\sin(\omega(y^{-1}x))}(1-e_2(y^{-1}x)^2),\nonumber \\
X_3^xe_2(x,y)&=-\frac{e_2(y^{-1}x)e_3(y^{-1}x)(1+\cos(\omega(y^{-1}x)))}{2\sin(\omega(y^{-1}x))}-\frac{e_1(y^{-1}x)}{2},\\
X_1^xe_3(x,y)&=-\frac{e_1(y^{-1}x)e_3(y^{-1}x)(1+\cos(\omega(y^{-1}x)))}{2\sin(\omega(y^{-1}x))}-\frac{e_2(y^{-1}x)}{2},\nonumber \\
X_2^xe_3(x,y)&=-\frac{e_2(y^{-1}x)e_3(y^{-1}x)(1+\cos(\omega(y^{-1}x)))}{2\sin(\omega(y^{-1}x))}+\frac{e_1(y^{-1}x)}{2},\nonumber \\
X_3^xe_3(x,y)&=\frac{1+\cos(\omega(y^{-1}x))}{2\sin(\omega(y^{-1}x))}(1-e_3(y^{-1}x)^2).\nonumber
\end{align}

Observe that we have
\begin{align}\label{sin:lowest}
\sin(\omega) &\geq \frac{2}{\pi}\omega, \quad \text{ for } \omega \in (0,\pi/2].
\end{align}
Thus for $\omega \in (\frac{\pi}{2(N+1)},\frac{\pi}{2}]$ we simply estimate
\begin{equation*}
\left|\frac{1+\cos(\omega)}{\sin(\omega)}\tilde{\sigma}_N'(\omega)\right| \leq \pi\frac{|\tilde{\sigma}_N'(\omega)|}{|\omega|}\leq 2(N+1)|\tilde{\sigma}_N'(\omega)|.
\end{equation*}
For $\omega \in (\frac{\pi}{2},\pi]$ we have
\begin{equation*}
\left|\frac{1+\cos(\omega)}{\sin(\omega)}\tilde{\sigma}_N'(\omega)\right| \leq 2\left(1-\frac{x}{\pi}\right)|\tilde{\sigma}_N'(\omega)|\leq|\tilde{\sigma}_N'(\omega)|.
\end{equation*}

Since $|e_ie_j|\leq \frac{1}{2}$ and $N\geq 2s \geq 12$, we have the estimate
\begin{align*}
|X_j^xX_i^y\sigma_N(x,y)| &\leq
\left(\frac{1}{2}(N+1)+\frac{1}{2}\right)|\tilde{\sigma}_N'(\omega(y^{-1}x))|+\frac{1}{2}|\tilde{\sigma}_N^{''}(\omega(y^{-1}x))|,\\
&\leq (N+1)|\tilde{\sigma}_N'(\omega(y^{-1}x))|+\frac{1}{2}|\tilde{\sigma}_N^{''}(\omega(y^{-1}x))|.
\end{align*}
Now we use the localization result of Theorem \ref{trig:est} together with $c_{1,s}\leq \frac{1}{2}c_{2,s}$ to derive
\begin{equation*}
|X_j^xX_i^y\sigma_N(x,y)|\leq \frac{c_{1,s}+\frac{1}{2}c_{2,s}}{(N+1)^{s-2}\omega(y^{-1}x)^s}\leq\frac{c_{2,s}}{(N+1)^{s-2}\omega(y^{-1}x)^s}.
\end{equation*}
If $i=j$, we have 
\begin{align*}
|X_i^xX_i^y\sigma_N(x,y)| &\leq (1-e_i^2)(N+1)|\tilde{\sigma}_N^{'}(\omega(y^{-1}x))|+e_i^2|\tilde{\sigma}_N^{''}(\omega(y^{-1}x))|,\\
&\leq \frac{(1-e_i^2)c_{1,s}+e_i^2c_{2,s}}{(N+1)^{s-2}\omega(y^{-1}x)^s}\leq\frac{c_{2,s}}{(N+1)^{s-2}\omega(y^{-1}x)^s}.  
\end{align*}
For $x=y$ we have  
\begin{equation*}
X_i^xX_i^y\sigma_N(x,x)=-\tilde{\sigma}_N^{''}(0), \quad X_j^xX_i^y\sigma_N(x,x)=0.
\end{equation*}
\end{proof}
The shown bounds are useful for estimating the entries of the interpolation matrix. In addition, we need localization estimates for the entries of the Hessian matrix.
We have to distinguish between two cases, namely $\omega(y^{-1}x))$ is well separated from zero, covered by Lemma \ref{local:kernelall2}, and $\omega(y^{-1}x))$ approaches zeros, which is handled in Lemma \ref{lip:est}.
\begin{Lemma}\label{local:kernelall2}
For $s \in 2\N$, $s \geq 6$, $N\geq 2s$, $\omega(y^{-1}x)\geq \frac{\pi}{2(N+1)}$,  $i,j,n$ pairwise different with the sign convention 
\begin{equation*}
X_j^xe_i=-e_ie_j\left(\frac{1+\cos(\omega)}{2\sin(\omega)}\right)\pm\frac{e_n}{2},
\end{equation*}
see proof of Theorem \ref{local:kernelall}, we have
\begin{align*}
|X_i^{x}X_i^{x}X_k^{y}\sigma_N(x,y)|&\leq \frac{1.2 \cdot c_{3,s}}{(N+1)^{s-3}\omega(y^{-1}x)^s},\\
\left|X_j^xX_i^x\sigma_N(x,y)\mp \frac{1}{2}X_n^x\sigma_N(x,y)\right| &\leq \frac{c_{2,s}}{(N+1)^{s-2}\omega(y^{-1}x)^s},\\
\left|X_j^xX_i^xX_k^y\sigma_N(x,y)\mp \frac{1}{2}X_n^xX_k^y\sigma_N(x,y)\right|&\leq \frac{1.2 \cdot c_{3,s}}{(N+1)^{s-3}\omega(y^{-1}x)^s}, 
\end{align*}
for $k=i, j, n$. 
\end{Lemma}

\begin{Lemma}\label{lip:est}
For $s \in 2\N$, $s \geq 6$, $N\geq 2s$, $\omega(y^{-1}x)\leq \frac{\delta}{N+1}$, $0 \leq \delta \leq \frac{\pi}{2}$ we have the following Lipschitz-type estimates for $i,j,n$ pairwise different
\begin{align*}
\left|X_i^xX_i^x\sigma_N(x,y)-\tilde{\sigma}_N^{''}(0)\right|&\leq \frac{\tilde{d}_{s}}{2}(N+1)^2\delta^2,\\
\left|X_i^xX_i^xX_k^y\sigma_N(x,y)\right|&\leq\tilde{d}_{s}\left((N+1)^3\delta+\frac{1}{4}(N+1)^2\delta^2\right)+\frac{\tilde{c}_s}{4}(N+1)\delta,
\end{align*} 
and with the sign convention such that
\begin{equation*}
X_j^xe_i=-e_ie_j\left(\frac{1+\cos(\omega)}{2\sin(\omega)}\right)\pm\frac{e_n}{2},
\end{equation*}
see proof of Theorem \ref{local:kernelall}, we have 
\begin{align*}
\left|X_j^xX_i^x\sigma_N(x,y)\mp \frac{1}{2}X_n^x\sigma_N(x,y)\right| &\leq \frac{\tilde{d}_{s}}{4}(N+1)^2\delta^2,\\
\left|X_j^xX_i^xX_k^y\sigma_N(x,y)\mp \frac{1}{2}X_n^xX_k^y\sigma_N(x,y)\right|&\leq\tilde{d}_{s}\left((N+1)^3\delta+\frac{1}{4}(N+1)^2\delta^2\right)+\frac{\tilde{c}_s}{4}(N+1)\delta,
\end{align*}
for $k=i,j,n$.
\end{Lemma}
The proofs of these two Lemmas are rather technical and can be found in Appendix \ref{app:proofs}. The last Lemma of this section provides bounds for summing up off-diagonal entries of the interpolation and the Hessian matrix.

\begin{Lemma}\label{offdiag:decay}
Let $x_j \in \mathcal{C}$, where $\mathcal{C}$ obeys a separation condition of $\rho(\mathcal{C})\geq \frac{\nu}{N+1}$ with $\nu \geq \pi$, and let $x \in SO(3)$ such that $d(x,x_j) \leq \varepsilon \frac{\nu}{N+1}$, for $0\leq \varepsilon \leq 1/2$. Then for any $s \geq 6$ and $N \geq 2s$,  $i,j,n$ pairwise different  
\begin{align*}
&\sum_{x_i \in \mathcal{C}\setminus x_j}|\sigma_N(x,x_i)| \leq \frac{C_{0,s}a_{\varepsilon}}{\nu^{s}},\\
&\sum_{x_i \in \mathcal{C}\setminus x_j}|X_n^y\sigma_N(x,x_i)| \leq \frac{C_{1,s}a_{\varepsilon}(N+1)}{\nu^{s}},\\
&\sum_{x_i \in \mathcal{C}\setminus x_j}|X_i^xX_n^y\sigma_N(x,x_i)| \leq\frac{C_{2,s}a_{\varepsilon}(N+1)^2}{\nu^{s}},\\ 
&\sum_{x_i \in \mathcal{C}\setminus x_j}|X_i^xX_i^xX_n^y\sigma_N(x,x_i)| \leq \frac{1.2C_{3,s}a_{\varepsilon}(N+1)^3}{\nu^{s}},\\
&\sum_{x_i \in \mathcal{C}\setminus x_j}\left|X_j^xX_i^x\sigma_N(x,y)\mp \frac{1}{2}X_n^x\sigma_N(x,y)\right| \leq \frac{C_{2,s}a_{\varepsilon}(N+1)^2}{\nu^{s}},\\
&\sum_{x_i \in \mathcal{C}\setminus x_j}\left|X_j^xX_i^xX_k^y\sigma_N(x,y)\mp \frac{1}{2}X_n^xX_k^y\sigma_N(x,y)\right|\leq \frac{1.2C_{3,s}a_{\varepsilon}(N+1)^2}{\nu^{s}}, \quad \text{ for } k=j,i,n,
\end{align*}
with $C_{i,s}=124c_{i,s}\zeta(s-2)$, where $c_{i,s}$ are the constants in Theorem \ref{local:kernelall} resp. Lemma \ref{local:kernelall2}, and $a_{\varepsilon}=\min\{\frac{27}{124}(1-\varepsilon)^{-s}+1,(1-\varepsilon)^{-s}\}$. Here $\zeta$ denotes the \emph{Riemannian Zeta function}.
\end{Lemma}
\begin{proof}
For $x \in SO(3)$, with $d(x,x_j) \leq \varepsilon \frac{\nu}{N+1}$ for some $x_j \in \mathcal{C}$, we define the ring about $x$ by
\begin{equation*}
\mathcal{S}_{m}:=\{y \in SO(3): \frac{\nu m}{N+1} \leq d(x,y) \leq \frac{\nu(m+1)}{N+1}\},
\end{equation*}
for $m \in \N$. By definition we have $\mathcal{S}_{m} = \emptyset$ for $\frac{m\nu}{N+1} > \pi$. 
Moreover, as shown in \cite{Schmid:2009} we can estimate the number of elements in the intersection of $\mathcal{S}_{m}$ with the set $\mathcal{C}\setminus\{x_j\}$ for $m\geq 1$ by
\begin{equation*}
\operatorname{card}(\mathcal{C}\setminus\{x_j\} \cap \mathcal{S}_{m}) \leq 48m^2+48m+28\leq 124 m^2.
\end{equation*}
Using the same technique as in \cite{Schmid:2009}, we have for $m=0$
\begin{equation*}
\operatorname{card}(\mathcal{C}\setminus\{x_j\}\cap \mathcal{S}_0) \leq 27.
\end{equation*}
Since $d(x,x_j) \leq \varepsilon \frac{\nu}{N+1}$, we have $d(x,x_i) \geq \frac{(1-\varepsilon)\nu }{N+1}$ for $x_i \in \mathcal{C}\setminus\{x_j\}\cap \mathcal{S}_0$. Using this and the locality result \eqref{local:kernelall}, we can estimate for $s\geq 4$
\begin{align*}
\sum_{x_i \in \mathcal{C}\setminus x_j}|\sigma_N(x,x_i)| &\leq\sum_{x_i \in (\mathcal{C}\setminus x_j) \cap \mathcal{S}_{0}} \frac{c_{0,s}}{((N+1)d(x,x_i))^s}+ \sum_{m=1}^{\infty}\sum_{x_i \in (\mathcal{C}\setminus x_j) \cap \mathcal{S}_{m}} \frac{c_{0,s}}{((N+1)d(x,x_i))^s},\\
&\leq\frac{27c_{0,s}(1-\varepsilon)^{-s}}{\nu^s}+124c_{0,s}\sum_{m=1}^{\infty}\frac{m^2}{(m\nu)^s},\\
&\leq \frac{27c_{0,s}(1-\varepsilon)^{-s}}{\nu^s}+\frac{124c_{0,s}}{\nu^{s}}\sum_{m=1}^{\infty}\frac{1}{m^{s-2}},\\
&\leq \frac{(27(1-\varepsilon)^{-s}+124)c_{0,s}\zeta(s-2)}{\nu^s},
\end{align*}
where the last inequality follows by the definition of the Zeta function.

On the other hand, we can define the rings around $x_j$ again by
 \begin{equation*}
\tilde{\mathcal{S}}_{m}:=\{y \in SO(3): \frac{(1-\varepsilon)\nu m}{N+1} \leq d(x_j,y) \leq \frac{(1-\varepsilon)\nu(m+1)}{N+1}\},
\end{equation*}.

Since $d(x,x_j) \leq \varepsilon \frac{\nu}{N+1}$, we have $d(x,x_j) \leq \varepsilon d(x_i,x_j)$ for $x_i \in (\mathcal{C} \setminus{x_j}) \cap \tilde{\mathcal{S}}_{m}$ and therefor $d(x,x_i) \geq d(x_i,x_j)-d(x,x_j) \geq \frac{(1-\varepsilon)\nu m}{N+1}$.
Using this and the locality result \eqref{local:kernelall} we can estimate for $s\geq 4$
\begin{align*}
\sum_{x_i \in \mathcal{C}\setminus x_j}|\sigma_N(x,x_i)| &\leq \sum_{m=1}^{\infty}\sum_{x_i \in (\mathcal{C}\setminus x_j) \cap \tilde{\mathcal{S}}_{m}} \frac{c_{0,s}}{((N+1)d(x,x_i))^s},\\
&\leq124c_{0,s}\sum_{m=1}^{\infty}\frac{m^2}{(1-\varepsilon)^s(m\nu)^s},\\
&\leq \frac{124c_{0,s}}{(1-\varepsilon)^s\nu^{s}}\sum_{m=1}^{\infty}\frac{1}{m^{s-2}}=\frac{124c_{0,s}\zeta(s-2)}{(1-\varepsilon)^s\nu^{s}},
\end{align*}

The estimations for the derivatives are derived in the same way using the corresponding locality results of \eqref{local:kernelall}. 
\end{proof}

\section{Construction of the Dual Certificate}\label{sec:dualconstr}
Suppose we are given a set $\mathcal{C}=\{x_1,\dots,x_M\}$ that satisfies for $\nu \geq \pi$ the separation condition
\begin{equation}\label{sep:cond}
\rho(\mathcal{C})= \min_{x_i,x_j \in \mathcal{C}, x_i \neq x_j}d(x_i,x_j) \geq \frac{\nu}{N+1}.
\end{equation}
The proportional factor $\nu$ is called \emph{super-resolution factor} and in this section we show that $\nu=36$ ensures the existence of a dual certificate, i.e. a function $q \in \Pi_N$ that fulfills the conditions of Theorem \ref{exact:recdual}. 

\subsection{Solution of the Interpolation problem}
We wish to find a generalized polynomial $q$ of degree at most $N$ such that
\begin{align*}
q(x_j)&=u_j,\\
X_1q(x_j)&=X_2q(x_j)=X_3q(x_j)=0,
\end{align*}
for $j=1,\dots,M$. Thus we would like to perform a Hermite type interpolation. 
To find a solution to the Hermite interpolation problem in the space $\Pi_N$ we determine coefficients $\alpha_{j,0}, \alpha_{j,1},  \alpha_{j,2},  \alpha_{j,3}$ for $j=1,\dots,M$ in the kernel expansion
\begin{align*}
q(x)=\sum_{j=1}^M\alpha_{j,0}\sigma_N(x,x_j)+ \alpha_{j,1}X_1^y\sigma_N(x,x_j)+  \alpha_{j,2}X_2^y\sigma_N(x,x_j)+  \alpha_{j,3}X_3^y\sigma_N(x,x_j),
\end{align*}
satisfying
\begin{equation}\label{interpol:eq}
K\alpha:=
\begin{pmatrix}
\sigma_N & X^x_1\sigma_N & X^x_2\sigma_N & X^x_3\sigma_N\\
X^y_1\sigma_N & X^x_1X^y_1\sigma_N & X^x_2X^y_1\sigma_N & X^x_3X^y_1\sigma_N \\
X^y_2\sigma_N & X^x_1X^y_2\sigma_N & X^x_2X^y_2\sigma_N & X^x_3X^y_2\sigma_N \\
X^y_3\sigma_N & X^x_1X^y_3\sigma_N & X^x_2X^y_3\sigma_N & X^x_3X^y_3\sigma_N
\end{pmatrix}
\begin{pmatrix}
\alpha_0\\
\alpha_1\\
\alpha_2\\
\alpha_3
\end{pmatrix}
=
\begin{pmatrix}
u\\
0\\
0\\
0\\
\end{pmatrix},
\end{equation}
where the entries in the matrix corresponds to blocks of the form $\sigma_N=(\sigma_N(x_i,x_j))_{i,j=1}^M$ and in the same way for the derivatives. The blocks in the vectors are given by $\alpha_k=(\alpha_{k,j})_{j=1}^M$, for $k=0,1,2,3$, and $u=(u_j)_{j=1}^M$. In the case this matrix is invertible, we have that $q$ satisfies the Hermite interpolation conditions. Moreover, by construction of the kernel $\sigma_N$, the function $q$ is always a polynomial of degree at most $N$. For abbreviation 
we write
\begin{equation*}
\sigma_{ij}=X_i^{x}X_j^{y}\sigma_N, \quad i,j=1,\dots,3.
\end{equation*}

So we have to show that the block matrix 
\begin{equation*}
K=
\begin{pmatrix}
K_0 & \tilde{K_{1}}\\
K_{1} & K_2
\end{pmatrix},
\end{equation*}
with blocks given by 
\begin{align*}
K_0&=\sigma_{00}=\sigma_N,\\
K_1&=
\begin{bmatrix}
\sigma_{01}&\sigma_{02}&\sigma_{03}
\end{bmatrix}^T=
\begin{bmatrix}
X_1^y\sigma_N & X_2^y\sigma_N &X_3^y\sigma_N
\end{bmatrix}^T,\\
\tilde{K_1}&=
\begin{bmatrix}
\sigma_{10}&\sigma_{20}&\sigma_{30}
\end{bmatrix}
=
\begin{bmatrix}
X_1^x\sigma_N & X_2^x\sigma_N &X_3^x\sigma_N
\end{bmatrix},\\
K_2&=
\begin{bmatrix}
\sigma_{11} & \sigma_{21} & \sigma_{31} \\
\sigma_{12} & \sigma_{22} & \sigma_{32} \\
\sigma_{13} & \sigma_{23} & \sigma_{33}
\end{bmatrix}=
\begin{bmatrix}
X^x_1X^y_1\sigma_N & X^x_2X^y_1\sigma_N & X^x_3X^y_1\sigma_N \\
X^x_1X^y_2\sigma_N & X^x_2X^y_2\sigma_N & X^x_3X^y_2\sigma_N \\
X^x_1X^y_3\sigma_N & X^x_2X^y_3\sigma_N & X^x_3X^y_3\sigma_N
\end{bmatrix}
\end{align*}
is invertible. To do this, we use an two step block inversion to show that both the matrix $K_2$ and its Schur complement $K / K_2=K_0-\tilde{K_1}K_2^{-1}K_1$ are invertible. 

To show the invertibility of $K_2$, we split up $K_2$ in the first step furthermore into blocks as
\begin{equation*}
K_2=\begin{pmatrix}
K_{2,0}& \tilde{K}_{2,1}\\
K_{2,1}& K_{2,2}
\end{pmatrix},
\end{equation*}
with
\begin{align*}
K_{2,0}&=\sigma_{11},\\
K_{2,1}&=\begin{bmatrix}
\sigma_{12} & \sigma_{13}
\end{bmatrix}^{T},\\
\tilde{K}_{2,1}&=\begin{bmatrix}
\sigma_{21} & \sigma_{31}
\end{bmatrix},\\
K_{2,2}&=\begin{bmatrix}
\sigma_{22} & \sigma_{32} \\
\sigma_{23} & \sigma_{33}
\end{bmatrix}.
\end{align*}
This shows that $K_2$ is invertible, if $K_{2,2}$ is invertible and its Shur complement in $K_2$ given by 
\begin{equation*}
S=K_{2} / K_{2,2} =K_{2,0}-\tilde{K}_{2,1}K_{2,2}^{-1}K_{2,1}
\end{equation*}
is invertible. 
For the invertibility of $K_{2,2}$, we proof the invertibility of
\begin{equation*}
\sigma_{33}=X^x_3X^y_3\sigma_N
\end{equation*}
and its Schur complement in $K_{2,2}$ given by
\begin{equation*}
T=K_{2,2}/\sigma_{33}=\sigma_{22}-\sigma_{32}\left(\sigma_{33}\right)^{-1}\sigma_{23}.
\end{equation*}
Having this, we go backwards determining the inverse of $K_2$ and in the end of $K$.
For this purpose, we use that a matrix $A$ is invertible if
\begin{equation*}
\|I-A\|_{\infty} < 1,
\end{equation*} 
where $\|A\|_{\infty}=\max_{i}\sum_{j}|a_{i,j}|$. In this case the norm of the inverse is bounded by
\begin{equation*}
\|A^{-1}\|_{\infty}\leq \frac{1}{1-\|I-A\|_{\infty}}.
\end{equation*}

In the following Lemma we bound the norms of the corresponding entries in the kernel matrix $K$.
\begin{Lemma}\label{sup:bounds}
If the separation condition \eqref{sep:cond} is satisfied, we have for any $s \geq 6$ even, $N\geq2s$ with $C_{i,s}=124c_{i,s}\zeta(s-2)$ and $c_s=\frac{0.999}{2(s+1)}$ the estimations
\begin{align*}
&\left\|I-\sigma_{00}\right\|_{\infty}\leq \frac{C_{0,s}}{\nu^s},\quad \left\|\sigma_{00}^{-1}\right\|_{\infty}\leq \frac{1}{1-\frac{C_{0,s}}{\nu^s}},\\
&\left\|\sigma_{0i} \right\|_{\infty}, \quad \left\|\sigma_{i0} \right\|_{\infty} \leq \frac{C_{1,s}(N+1)}{\nu^s}, \quad\left\|\sigma_{ij}\right\|_{\infty} \leq \frac{C_{2,s}(N+1)^2}{\nu^s}, \quad \text{ for } i \neq j, i,j \neq0,\\
&\left\|-\tilde{\sigma}_N^{''}(0)I-\sigma_{ii}\right\|_{\infty} \leq \frac{C_{2,s}(N+1)^2}{\nu^s}, \quad  \left\| \sigma_{ii}^{-1}\right\|_{\infty} \leq \frac{1}{c_s(N+1)^2\left(1-\frac{C_{2,s}}{c_s\nu^s}\right)}.
\end{align*}
\end{Lemma}
\begin{proof}
The proof follows directly from applying Lemma \ref{offdiag:decay} together with the bound for $|\tilde{\sigma}_N^{''}(0)|$ given in Lemma \ref{2deriv:lowbound}.
\end{proof}

\begin{Lemma}\label{interpol:solution}
Suppose the separation condition \eqref{sep:cond} is satisfied, such that for $s \geq 6$ even and $N\geq2s$, there is a constant $b> 3 + \frac{c_s}{4}$, with
\begin{equation}\label{grow:qout}
\nu^s > b \frac{C_{2,s}}{c_s},
\end{equation}
where the constant $c_s$ is given in Lemma \ref{2deriv:lowbound} and $C_{2,s}$ in Lemma \ref{offdiag:decay}. Then the interpolation problem \eqref{interpol:eq} has a unique solution, such that the coefficients obey
\begin{equation*}
\|\alpha_0\|_{\infty}\leq 1+\frac{c_s}{4(b-3)-c_s}, \quad \|\alpha_j\|_{\infty} \leq \frac{2}{(4(b-3)-c_s)(N+1)},\quad  j=1,2,3.
\end{equation*}
Moreover, if $u_i=1$ we have the bound
\begin{equation*}
\alpha_{0,i}\geq 1-\frac{c_s}{4(b-3)-c_s}.
\end{equation*}
\end{Lemma}
\begin{proof}
In this proof the quotient $\frac{C_{2,s}}{c_s\nu^s}$ appears quite often, so we will denote it for abbreviation by
\begin{equation}
a_1:=\frac{C_{2,s}}{c_s\nu^s}.
\end{equation} 
Using Lemma \ref{sup:bounds} we have that
\begin{equation*}
\left\| \sigma_{33}^{-1}\right\|_{\infty} \leq \frac{1}{c_s(N+1)^2\left(1-a_1\right)}
\end{equation*}
and 
\begin{align*}
\left\|\tilde{\sigma}_N^{''}(0)I-K_{2,2}/\sigma_{33}\right\|_{\infty}&\leq \left\|\tilde{\sigma}_N^{''}(0)I-\sigma_{22}\right\|_{\infty}+
\left\|\sigma_{32}\right\|_{\infty}\left\|\sigma_{33}^{-1}\right\|_{\infty}\left\|\sigma_{23}\right\|_{\infty},\\
&\leq \frac{C_{2,s}(N+1)^2}{\nu^s}\left(1+\frac{C_{2,s}}{c_s\nu^s-C_{2,s}} \right). 
\end{align*}
This means
\begin{equation*}
\left\|I-\frac{K_{2,2}/\sigma_{33}}{\tilde{\sigma}_N^{''}(0)}\right\|_{\infty} \leq \frac{1}{|\tilde{\sigma}_N^{''}(0)|}\frac{C_s(N+1)^2}{\nu^s}\left(1+\frac{C_s}{c_s\nu^s-C_s} \right).
\end{equation*}
If the expression on the right hand side is smaller than $1$, which is the case if 
\begin{equation}\label{constr:a2}
a_2:=\frac{C_{2,s}}{c_s\nu^s}\left(1+\frac{C_{2,s}}{c_s\nu^s-C_{2,s}}\right)= \frac{C_{2,s}}{c_s\nu^s-C_{2,s}}=\frac{a_1}{1-a_1}<1,
\end{equation}
or equivalently
\begin{equation}
\nu^s > 2 \frac{C_{2,s}}{c_s},
\end{equation}
 we have
\begin{equation*}
\left\|\left(K_{2,2}/\sigma_{33}\right)^{-1}\right\|_{\infty}\leq \frac{1}{c_s(N+1)^2(1-a_2)}.
\end{equation*}

This shows the invertibility of $K_{2,2}$. Observe that we have with $T=K_{2,2}/\sigma_{33}$ the representation
\begin{equation}\label{inverse:22}
\left(K_{2,2}\right)^{-1}=
\begin{pmatrix}
T^{-1} & -T^{-1}\sigma_{32}\left(\sigma_{33}\right)^{-1}\\
-\left(\sigma_{33}\right)^{-1}\sigma_{23}T^{-1} & \left(\sigma_{33}\right)^{-1} + \left(\sigma_{33}\right)^{-1} \sigma_{23}T^{-1}\sigma_{32}\left(\sigma_{33}\right)^{-1}
\end{pmatrix}.
\end{equation}

 In the next step we show the invertibility of the Schur complement of $K_{2,2}$ in $K_2$, which is given by 
$K_2 / K_{2,2}=K_{2,0}-\tilde{K}_{2,1}K_{2,2}^{-1}K_{2,1}$. By the quotient formula for Schur
complements we can express $K_2/K_{2,2}$ as
\begin{equation*}
K_2 / K_{2,2}=(K_2/\sigma_{33})/(K_{2,2}/\sigma_{33}).
\end{equation*}  
Thus, we have to look at the matrix $K_2/\sigma_{33}$. Using the alternative partition of $K_2$, given by
\begin{equation*}
K_2=
\begin{pmatrix}
A & B \\
C & \sigma_{33}
\end{pmatrix},
\end{equation*}
with
\begin{align*}
A&=\begin{bmatrix}
\sigma_{11} & \sigma_{21} \\
\sigma_{12} & \sigma_{22} 
\end{bmatrix},\\
B&=\begin{bmatrix}
\sigma_{31} & \sigma_{32}  
\end{bmatrix}^T,\\
C&=\begin{bmatrix}
\sigma_{13} & \sigma_{23}  
\end{bmatrix},
\end{align*}
shows that we have
\begin{align}\label{calK_21}
K_2 / \sigma_{33}&= A-B\left(\sigma_{33}\right)^{-1}C,\nonumber \\
&=\begin{pmatrix}
\sigma_{11}-\sigma_{31}\left(\sigma_{33}\right)^{-1}\sigma_{13} &\sigma_{21}-\sigma_{31}\left(\sigma_{33}\right)^{-1}\sigma_{23}\\
\sigma_{12}-\sigma_{32}\left(\sigma_{33}\right)^{-1}\sigma_{13} &\sigma_{22}-\sigma_{32}\left(\sigma_{33}\right)^{-1}\sigma_{23}
\end{pmatrix},\nonumber \\
&=:\begin{pmatrix}
\mathcal{K}_{2,0} &\tilde{\mathcal{K}}_{2,1}\\
\mathcal{K}_{2,1} &K_{2,2}/\sigma_{33}
\end{pmatrix}.
\end{align}

This means
\begin{align*}
K_2 / K_{2,2}&=(K_2/\sigma_{33})/(K_{2,2}/\sigma_{33})\\
&=\sigma_{11}-\sigma_{31}\left(\sigma_{33}\right)^{-1}\sigma_{13}-\tilde{\mathcal{K}}_{2,1}\left(K_{2,2}/\sigma_{33}\right)^{-1}\mathcal{K}_{2,1}.
\end{align*}

So we can estimate using Lemma \ref{sup:bounds}
\begin{align*}
\left\|-\tilde{\sigma}_N^{''}(0)I-K_2/K_{2,2}\right\|_{\infty}&\leq \left\|\tilde{\sigma}_N^{''}(0)I-\sigma_{11}\right\| + \left\|\sigma_{31} \left(\sigma_{33}\right)^{-1}\sigma_{13} \right\|_{\infty}\\
&\quad +\left\|\tilde{\mathcal{K}}_{2,1}\left(K_{2,2}/\sigma_{33}\right)^{-1}\mathcal{K}_{2,1}\right\|_{\infty},\\
&\leq c_s a_2\left(1+\frac{a_2}{1-a_2} \right)=c_s\frac{a_2}{1-a_2},
\end{align*}
with $a_2$ given in \eqref{constr:a2}.
This means, together with the bound derived for $\left(K_{2,2}/\sigma_{33}\right)^{-1}$,
\begin{align*}
\left\|I-\frac{K_2/K_{2,2}}{-\tilde{\sigma}_N^{''}(0)}\right\|_{\infty}&\leq \frac{a_2}{1-a_2} = \frac{a_1}{1-2a_1},
\end{align*}
with $a_1=\frac{C_{2,s}}{c_s\nu^s}$, and $K_2/K_{2,2}$ is invertible if
\begin{align*}
a_3:=\frac{a_1}{1-2a_1} < 1
\end{align*}
or equivalently
\begin{equation*}
\nu^s > 3\frac{C_{2,s}}{c_s}.
\end{equation*}
We have
\begin{equation*}
\left\|\left(K_{2}/K_{2,2}\right)^{-1} \right\|_{\infty} \leq \frac{1}{c_s(N+1)^2(1-a_3)}.
\end{equation*}
This shows the invertibility of $K_2$. If we denote $S=K_2/K_{2,2}$, then the inverse is given by
\begin{align}\label{inverse:2}
K_2^{-1}&=\begin{pmatrix}
S^{-1} & -S^{-1}\tilde{K}_{2,1}K_{2,2}^{-1}\\
-K_{2,2}^{-1}K_{2,1}S^{-1} & K_{2,2}^{-1}+K_{2,2}^{-1}K_{2,1}S^{-1}\tilde{K}_{2,1}K_{2,2}^{-1} 
\end{pmatrix}.
\end{align}

 In the last step we apply the same procedure to show the invertibility of $R=K/K_2$. So, as seen before, we use the quotient rule
\begin{equation*}
K/K_2=(K/K_{2,2})/(K_{2}/K_{2,2}).
\end{equation*} 
To calculate $K/K_{2,2}$ we split $K$ into blocks as
\begin{equation*}
K=
\begin{pmatrix}
A & B \\
C & K_{2,2}
\end{pmatrix},
\end{equation*}
with
\begin{align*}
A&=\begin{bmatrix}
\sigma_{00} & \sigma_{10} \\
\sigma_{01} & \sigma_{11} 
\end{bmatrix},\\
B&=\begin{bmatrix}
\sigma_{20} &\sigma_{30}\\
\sigma_{21} &\sigma_{31}   
\end{bmatrix},\\
C&=\begin{bmatrix}
\sigma_{02} &\sigma_{12}\\
\sigma_{03} &\sigma_{13}  
\end{bmatrix},
\end{align*}
which leads to
\begin{equation*}
K/K_{2,2}=A-B\left(K_{2,2}\right)^{-1}C.
\end{equation*}

A lengthy calculation shows that we can write
\begin{equation*}
K/K_{2,2}=A-B\left(K_{2,2}\right)^{-1}C
=\begin{pmatrix}
\mathcal{K}_0 & \tilde{\mathcal{K}}_1\\
\mathcal{K}_1 & K_2/K_{2,2}
\end{pmatrix},
\end{equation*}
with
\begin{align*}
\mathcal{K}_0&=\left(\sigma_{00}-\sigma_{30}\left(\sigma_{33}\right)^{-1}\sigma_{03}\right)-\tilde{\mathcal{C}}_{2,1}T^{-1}\mathcal{C}_{2,1},\\
\tilde{\mathcal{K}}_{1}&=\tilde{\mathcal{G}}_{2,1}-\tilde{\mathcal{C}}_{2,1}T^{-1}\mathcal{K}_{2,1},\\
\mathcal{K}_{1} &=\mathcal{G}_{2,1}-\tilde{\mathcal{K}}_{2,1}T^{-1}\mathcal{C}_{2,1},
\end{align*}
where $\mathcal{K}_{2,1}$, $\tilde{\mathcal{K}}_{2,1}$ are given by \eqref{calK_21} and
\begin{align*}
\mathcal{C}_{2,1}&=\sigma_{02}-\sigma_{32}\left(\sigma_{33}\right)^{-1}\sigma_{03},\\
\tilde{\mathcal{C}}_{2,1}&=\sigma_{20}-\sigma_{30}\left(\sigma_{33}\right)^{-1}\sigma_{23},\\
\mathcal{G}_{2,1}&=\sigma_{01}-\sigma_{31}\left(\sigma_{33}\right)^{-1}\sigma_{03},\\
\tilde{\mathcal{G}}_{2,1}&=\sigma_{10}-\sigma_{30}\left(\sigma_{33}\right)^{-1}\sigma_{13}.
\end{align*}
This yields
\begin{equation*}
K/K_{2}=\mathcal{K}_{0}-\tilde{\mathcal{K}}_1\left(K_2/K_{2,2}\right)^{-1}\mathcal{K}_1,
\end{equation*}
and therefore
\begin{equation*}
\|I-K/K_{2}\|_{\infty} \leq \|I - \mathcal{K}_{0}\|_{\infty}+\|\tilde{\mathcal{K}}_1\|_{\infty}\|\mathcal{K}_1\|_{\infty}\left\|\left(K_2/K_{2,2}\right)^{-1}\right\|_{\infty}.
\end{equation*}
Observe, that we have the bounds 
\begin{equation*}
\|\mathcal{C}_{2,1}\|_{\infty}, \|\tilde{\mathcal{C}}_{2,1}\|_{\infty}, \|\mathcal{G}_{2,1}\|_{\infty},\|\tilde{\mathcal{G}}_{2,1}\|_{\infty} \leq \frac{C_{1,s}(N+1)}{\nu^s}(1+a_2)=\frac{C_{1,s}(N+1)}{\nu^s}\frac{1}{1-a_1}.
\end{equation*}

Using this we have that
\begin{align*}
\|I-\mathcal{K}_0\|_{\infty} & \leq \|I-\sigma_{00}\|_{\infty}+ \|\sigma_{30}\|_{\infty}\left\|\sigma_{03}\right\|_{\infty}\left\|\left(\sigma_{33}\right)^{-1}\right\|_{\infty}+\|\mathcal{C}_{2,1}\|_{\infty}\|\tilde{\mathcal{C}}_{2,1}\|_{\infty}\|T^{-1}\|_{\infty},\\
&\leq \frac{C_{0,s}}{\nu^s}+\left(\frac{C_{1,s}}{\nu^s}\right)^2\frac{2}{c_s(1-2a_1)},
\end{align*}
and similarly
\begin{equation*}
\|\tilde{\mathcal{K}}_1\|_{\infty},\quad \|\mathcal{K}_1\|_{\infty} \leq \frac{C_{1,s}(N+1)}{\nu^s}\frac{1}{1-2a_1}.
\end{equation*}

This results in
\begin{equation*}
\|I-K/K_{2}\|_{\infty} \leq \frac{C_{0,s}}{\nu^s}+ \left(\frac{C_{1,s}}{\nu^s} \right)^2\frac{3}{c_s(1-3a_1)}.
\end{equation*}
Thus $K/K_2$ is invertible if 
\begin{equation*}
a_4:=\frac{C_{0,s}}{\nu^s}+ \left(\frac{C_{1,s}}{\nu^s} \right)^2\frac{3}{c_s(1-3a_1)} <1.
\end{equation*}
Because of $2C_{0,s}\leq C_{1,s}\leq \frac{C_{2,s}}{2}$, we have the bound
\begin{equation*}
a_4\leq \frac{1}{4}c_s\frac{a_1}{1-3a_1},
\end{equation*}
and $a_4<1$, if 
\begin{equation*}
\nu^s > (3+\frac{c_s}{4})\frac{C_{2,s}}{c_s}.
\end{equation*}
Since $\nu^s > b \frac{C_{2,s}}{c_s}$ with $b > 3+\frac{c_s}{4}$, this is true and we have 
\begin{align*}
a_4 &\leq \frac{c_s}{4(b-3)},\\
 \left\|\left(K/K_2\right)^{-1}\right\|_{\infty} &\leq \frac{1}{1-a_4}\leq 1+\frac{c_s}{4(b-3)-c_s}.
\end{align*}
This gives the invertibility of $K$. With $R=K/K_2$ we have
\begin{equation*}
K^{-1}=
\begin{pmatrix}
R^{-1} & -R^{-1}\tilde{K}_1K_2^{-1}\\
-K_{2}^{-1}K_{1}R^{-1} & K_{2}^{-1}+K_{2}^{-1}K_1R^{-1}\tilde{K}_1K_{2}^{-1} 
\end{pmatrix}.
\end{equation*}

In reference to the interpolation condition \eqref{interpol:eq}, we have the representation of the coefficients
\begin{equation*}
\alpha=\begin{pmatrix}
I\\
-K_2^{-1}K_1
\end{pmatrix}
\left(K/K_2\right)^{-1}u,
\end{equation*} 
so we have to calculate the product $-K_2^{-1}K_1$. Using the representations \eqref{inverse:2} and \eqref{inverse:22} of the inverse of $K_2$ resp. $K_{2,2}$, a lengthy calculation shows
\begin{align*}
\alpha_0&=\left(K/K_2\right)^{-1}u,\\
\alpha_1&=S^{-1}(\mathcal{G}_{2,1}-\tilde{\mathcal{K}}_{2,1}T^{-1}\mathcal{C}_{2,1})\left(K/K_2\right)^{-1}u,\\
&=S^{-1}(\mathcal{G}_{2,1}-\tilde{\mathcal{K}}_{2,1}T^{-1}\mathcal{C}_{2,1})\alpha_0,\\
\alpha_2&=T^{-1}\left(\mathcal{C}_{2,1}-\mathcal{K}_{2,1}S^{-1}(\mathcal{G}_{2,1}-\tilde{\mathcal{K}}_{2,1}T^{-1}\mathcal{C}_{2,1})\right)\left(K/K_2\right)^{-1}u,\\
&=T^{-1}(\mathcal{C}_{2,1}\alpha_0-\mathcal{K}_{2,1}\alpha_1),\\
\alpha_3&=\left(\sigma_{33}\right)^{-1}\Big[\sigma_{03}-\sigma_{23}T^{-1}\left(\mathcal{C}_{2,1}-\mathcal{K}_{2,1}S^{-1}(\mathcal{G}_{2,1}-\tilde{\mathcal{K}}_{2,1}T^{-1}\mathcal{C}_{2,1})\right)\\
&\quad -\sigma_{13}S^{-1}(\mathcal{G}_{2,1}-\tilde{\mathcal{K}}_{2,1}T^{-1}\mathcal{C}_{2,1})\Big]\left(K/K_2\right)^{-1}u,\\
&=\left(\sigma_{33}\right)^{-1}\left(\sigma_{03}\alpha_0-\sigma_{23}\alpha_2-\sigma_{13}\alpha_1\right).
\end{align*}
Together with the observation that $\frac{C_{1,s}}{C_{2,s}}\leq\frac{1}{2}$, we can estimate the norm of the coefficients by
\begin{align*}
\|\alpha_0\|_{\infty}&\leq \|\left(K/K_2\right)^{-1}\|_{\infty}\leq 1+\frac{c_s}{4(b-3)-c_s},\\
\|\alpha_1\|_{\infty}&\leq \|S^{-1}\|_{\infty}(\|\mathcal{G}_{2,1}\|_{\infty}+\|\tilde{\mathcal{K}}_{2,1}\|_{\infty}\|\mathcal{C}_{2,1}\|_{\infty}\|T^{-1}\|_{\infty})\|\alpha_0\|_{\infty},\\
&\leq \frac{2}{(4(b-3)-c_s)(N+1)},\\
\|a_2\|_{\infty}&\leq \|T^{-1}\|_{\infty}(\|\mathcal{C}_{2,1}\|_{\infty}\|\alpha_0\|_{\infty}+\|\mathcal{K}_{2,1}\|_{\infty}\|\alpha_1\|_{\infty}),\\
&\leq \frac{2}{(4(b-3)-c_s)(N+1)},\\
\|\alpha_3\|_{\infty} &\leq \left\|\left(\sigma_{33}\right)^{-1}\right\|_{\infty}\left(\|\sigma_{03}\|_{\infty}\|\alpha_0\|_{\infty}+\|\sigma_{23}\|_{\infty}\|\alpha_2\|_{\infty}+\|\sigma_{13}\|_{\infty}\|\alpha_1\|_{\infty}\right),\\
&\leq \frac{2}{(4(b-3)-c_s)(N+1)}.
\end{align*}
Moreover, if $u_i=1$ then we have the bound
\begin{align*}
\alpha_{0,i}&=\left(I-\left(I-\left(K/K_2\right)^{-1}\right)u\right)_i,\\
&= u_i-\left(\left(I-\left(K/K_2\right)^{-1}\right)u\right)_i,\\
&\geq 1-\|I-K/K_2\|_{\infty}\|\left(K/K_2\right)^{-1}\|,\\
&\geq 1-\frac{a_4}{1-a_4}\geq 1-\frac{c_s}{4(b-3)-c_s}.
\end{align*}
\end{proof}

\begin{corollary}\label{interpol:sol2}
Suppose the interpolation points $\mathcal{C}=\{x_1,\dots,x_M\}$ obey the separation condition
\begin{equation}\label{sep:cond2}
\min_{x_i\neq x_j} \omega(x_j^{-1}x_i) \geq \frac{36}{N+1}
\end{equation}
for $N \geq 20$. Then the interpolation problem has a unique solution $q \in \Pi_N$.
\end{corollary}
\begin{proof}
It can be checked that the condition \eqref{grow:qout} is fulfilled for the parameters $\nu=36$, $b=28$ and $s=8$. 
\end{proof}

\subsection{Bound for the Interpolant}
Next to the interpolation conditions, guaranteed by Corollary \ref{interpol:sol2}, we have to show that the interpolant $q$ fulfills the second assumption of \eqref{interpol:cond1} in Theorem \ref{exact:recdual}, namely $|q(x)|<1$ for $x$ not being an interpolation point. We split the proof into those $x\in SO(3)$, that are close to an interpolation point, which is Lemma \ref{defi:Hess}, and those that are well separated, governed by Lemma \ref{away:decay}.
\begin{Lemma}\label{defi:Hess}
Suppose the separation condition \eqref{sep:cond2} is satisfied for $N \geq 20$.
Then for all $x\in SO(3)$, such that there is a $x_m$ with $\omega(x^{-1}_mx)\leq \frac{\pi}{2(N+1)}$, we have for the interpolating function $q$ of Lemma \ref{interpol:solution} for $s=8$
\begin{equation*}
|q(x)| <1.
\end{equation*}
\end{Lemma}
\begin{proof}
The proof is based on a concavity respectively a convexity argument. We show that in the prescribed neighbourhood of an interpolation point $x_m$ the Hessian is negative definite in the case $u_m=1$ and positive definite in the case $u_m=-1$, using the Theorem of Gerschgorin. 
For this, first observe that the Hessian matrix for a function $f$ is given by 
\begin{equation}
Hf=
\begin{pmatrix}
X_1X_1f & X_1X_2f & X_1X_3f\\
X_2X_1f & X_2X_2f & X_2X_3f\\
X_3X_1f & X_3X_2f & X_3X_3f
\end{pmatrix}
-\frac{1}{2}\begin{pmatrix}
0 & X_3f & -X_2f\\
-X_3f & 0 & X_1f\\
X_2f & -X_1f & 0
\end{pmatrix}.
\end{equation}
If we apply this to the function constructed from the interpolation problem
\begin{equation*}
q(x)=\sum_{j=1}^M\alpha_{j,0}\sigma_N(x,x_j)+ \alpha_{j,1}X_1^y\sigma_N(x,x_j)+  \alpha_{j,2}X_2^y\sigma_N(x,x_j)+  \alpha_{j,3}X_3^y\sigma_N(x,x_j),
\end{equation*}
the diagonal entries of the matrix are given by
\begin{align*}
X_i^xX_i^xq(x)&=\sum_{j=1}^M\alpha_{j,0}X_i^xX_i^x\sigma_N(x,x_j)+ \alpha_{j,1}X_i^xX_i^xX_1^y\sigma_N(x,x_j)+  \alpha_{j,2}X_i^xX_i^xX_2^y\sigma_N(x,x_j)\\
&\quad+  \alpha_{j,3}X_i^xX_i^xX_3^y\sigma_N(x,x_j).
\end{align*}
For the estimates of the entries of the Hessian, we use Lemma \ref{local:kernelall2}, \ref{lip:est}, \ref{offdiag:decay}, \ref{interpol:solution} and \ref{2deriv:lowbound} with the following parameters $s=8$, $\nu=36$, $b=28$ and $\delta=\frac{\pi}{2}$.

We assume that $u_m=1$, since the estimates for $u_m=-1$ are completely analog. The first step is to show, that the diagonal entries are negative.
For this we estimate
\begin{align*}
X_i^xX_i^xq(x)&\leq \alpha_{0,m}X_i^xX_i^x\sigma_N(x,x_m)+ \sum_{n=1}^3\|\alpha_n\|_{\infty}\left|X_i^xX_i^xX_n^y\sigma_N(x,x_m)\right|\\
&\quad+\|\alpha_0\|_{\infty}\sum_{x_j\neq x_m}\left|X_i^xX_i^x\sigma_N(x,x_j)\right|
+\sum_{n=1}^3\|\alpha_n\|_{\infty}\sum_{x_j\neq x_m}\left|X_i^xX_i^xX_n^y\sigma_N(x,x_j)\right| 
\end{align*} 
The first term can be estimated using the bounds of Lemma \ref{2deriv:lowbound} and Lemma \ref{interpol:solution} by
\begin{align*}
\alpha_{0,m}X_i^xX_i^x\sigma_N(x,x_m)&=\alpha_{0,m}X_i^xX_i^x\sigma_N(x,x)+\alpha_{0,m}\left(X_i^xX_i^x\sigma_N(x,x_m)-X_i^xX_i^x\sigma_N(x,x)\right),\\
&=\alpha_{0,m}\tilde{\sigma}_N^{''}(0)+\alpha_{0,m}\left(X_i^xX_i^x\sigma_N(x,x_m)-\tilde{\sigma}_N^{''}(0)\right),\\
&\leq -c_s(N+1)^2\left(1-\frac{c_s}{4(b-3)-c_s}\right)\\
&\quad +\left(1+\frac{c_s}{4(b-3)-c_s}\right)\left(X_i^xX_i^x\sigma_N(x,x_m)-\tilde{\sigma}_N^{''}(0)\right).
\end{align*}
By Lemma \ref{lip:est} we have the estimation
\begin{equation*}
\left(X_i^xX_i^x\sigma_N(x,x_m)-\tilde{\sigma}_N^{''}(0)\right)\leq \frac{\tilde{d}_s}{2}(N+1)^2\delta^2, \quad \omega \in [0, \frac{\delta}{N+1}],
\end{equation*}
which yields 
\begin{align*}
\alpha_{0,m}X_i^xX_i^x\sigma_N(x,x_m)
&\leq -(N+1)^2\left(1+\frac{c_s}{4(b-3)-c_s}\right)c_s\left(1-\frac{c_s}{2(b-3)}-\frac{\delta^2\tilde{d}_s}{2c_s}\right).
\end{align*}
In addition, we have again using Lemma \ref{lip:est} and Lemma \ref{interpol:solution}
\begin{align*}
\sum_{n=1}^3\|\alpha_n\|_{\infty}\left|X_i^xX_i^xX_n^y\sigma_N(x,x_m)\right|
&\leq  \tilde{d}_s\frac{6(N+1)^2\delta}{4(b-3)-c_s}\left(1+\frac{\delta}{4(2s+1)}+\frac{\tilde{c_s}}{4\tilde{d_s}(2s+1)^2}\right),
\end{align*}
as well as using Lemma \ref{offdiag:decay} with $a_{\delta/\nu}=\min\{\frac{27}{124}(1-\frac{\delta}{\nu})^{-s}+1,(1-\frac{\delta}{\nu})^{-s}\}$ and the assumption $\nu^s>b\frac{C_{2,s}}{c_s}$
\begin{align*}
\|\alpha_0\|_{\infty}\sum_{x_j\neq x_m}\left|X_i^xX_i^x\sigma_N(x,x_j)\right|
&\leq \left(1+\frac{c_s}{4(b-3)-c_s}\right)\frac{c_sa_{\delta/\nu}(N+1)^2}{b}
\end{align*}
and
\begin{align*}
\sum_{n=1}^3\|\alpha_n\|_{\infty}\sum_{x_j\neq x_m}\left|X_i^xX_i^xX_n^y\sigma_N(x,x_j)\right|
&\leq\left(1+\frac{c_s}{4(b-3)-c_s}\right)\frac{27c_sa_{\delta/\nu}}{4(b-3)}\frac{(N+1)^2}{b}.
\end{align*}
Inserting the parameters we find
\begin{align*}
X_i^xX_i^xq(x)\leq -0.041\cdot(N+1)^2. 
\end{align*}

For the off-diagonal entries of the Hessian matrix we have
\begin{align*}
X_j^xX_i^xq\mp \frac{1}{2}X_n^xq&=\sum_{j=1}^M\alpha_{j,0}\left(X_j^xX_i^x\sigma_N(x,x_j)\mp\frac{1}{2}X_n^x\sigma_N(x,x_j)\right)\\
&\quad +\sum_{k=1}^3 \alpha_{j,k}\left(X_j^xX_i^xX_k^y\sigma_N(x,x_j)\mp\frac{1}{2}X_n^xX_k^x\sigma_N(x,x_j)\right)  
\end{align*}
and therefore we can estimate
\begin{align*}
|X_j^xX_i^xq\mp \frac{1}{2}X_n^xq|&\leq \|\alpha_{0}\|_{\infty}\left|X_j^xX_i^x\sigma_N(x,x_m)\mp\frac{1}{2}X_n^x\sigma_N(x,x_m)\right|\\
&\quad + \sum_{k=1}^3\|\alpha_k\|_{\infty}\left|X_j^xX_i^xX_k^y\sigma_N(x,x_m)\mp\frac{1}{2}X_n^xX_k^x\sigma_N(x,x_m)\right|\\
&\quad +\|\alpha_{0}\|_{\infty}\sum_{x_j\neq x_m}\left|X_j^xX_i^x\sigma_N(x,x_j)\mp\frac{1}{2}X_n^x\sigma_N(x,x_j)\right|\\
&\quad +\sum_{k=1}^3\|\alpha_{k}\|_{\infty}\sum_{x_j \neq x_m}\left|X_j^xX_i^xX_k^y\sigma_N(x,x_j)\mp\frac{1}{2}X_n^xX_k^x\sigma_N(x,x_j)\right|.
\end{align*}
Lemma \ref{lip:est} and Lemma \ref{interpol:solution} yield
\begin{align*}
\|\alpha_{0}\|_{\infty}\left|X_j^xX_i^x\sigma_N(x,x_m)\mp\frac{1}{2}X_n^x\sigma_N(x,x_m)\right|&\leq \left(1+\frac{c_s}{4(b-3)-c_s}\right)(N+1)^2\frac{\tilde{d}_s}{4}\delta^2
\end{align*}
as well as
\begin{align*}
&\sum_{k=1}^3\|\alpha_k\|_{\infty}\left|X_j^xX_i^xX_k^y\sigma_N(x,x_m)\mp\frac{1}{2}X_n^xX_k^x\sigma_N(x,x_m)\right|\\
&\leq \tilde{d}_s\frac{6(N+1)^2\delta}{4(b-3)-c_s}\left(1+\frac{\delta}{4(2s+1)}+\frac{\tilde{c_s}}{4\tilde{d_s}(2s+1)^2}\right).
\end{align*}
Furthermore we have applying Lemma \ref{local:kernelall2}
\begin{align*}
&\|\alpha_{0}\|_{\infty}\sum_{x_j\neq x_m}\left|X_j^xX_i^x\sigma_N(x,x_j)\mp\frac{1}{2}X_n^x\sigma_N(x,x_j)\right|\\
&\leq \left(1+\frac{c_s}{4(b-3)-c_s}\right)\frac{c_sa_{\delta/\nu}(N+1)^2}{b}
\end{align*}
and
\begin{align*}
&\sum_{k=1}^3\|\alpha_{k}\|_{\infty}\sum_{x_j \neq x_m}\left|X_j^xX_i^xX_k^y\sigma_N(x,x_j)\mp\frac{1}{2}X_n^xX_k^x\sigma_N(x,x_j)\right|\\
&\leq\left(1+\frac{c_s}{4(b-3)-c_s}\right)\frac{27c_sa_{\delta/\nu}}{4(b-3)}\frac{(N+1)^2}{b}.
\end{align*}
 Inserting the parameters results in
 \begin{align*}
 |X_j^xX_i^xq\mp \frac{1}{2}X_n^xq| \leq 0.01 \cdot (N+1)^2.
 \end{align*}
 Since
 \begin{equation*}
 |X_i^xX_i^xq(x)| > 2 |X_j^xX_i^xq\mp \frac{1}{2}X_n^xq|, 
 \end{equation*}
 and $X_i^xX_i^xq(x) < 0$ for $i=1,2,3$, we have that the Hessian matrix is negative definite at $x$.  
 
 The definitness shows $q(x) < 1$, to show $q(x)> -1$ observe that
 \begin{align*}
 q(x) &\geq \alpha_{0,m}\sigma_N(x,x_m)-\sum_{k=1}^3\|\alpha_k\|_{\infty}|X_k^y\sigma_N(x,x_m)|-\|\alpha_0\|_{\infty}\sum_{x_j \neq x_m}|\sigma_N(x,x_j)|\\
 &\quad - \sum_{k=1}^3\|\alpha_k\|_{\infty}\sum_{x_j \neq x_m}|X_k^x\sigma_N(x,x_j)|.
 \end{align*}
 We have, using the Taylor expansion of cosine at zero,
 \begin{align*}
 \sigma_N(x,x_m)&\geq 1-\frac{\tilde{c}_s}{2}\delta^2
 \end{align*}
 and
 \begin{align*}
|X_k^y\sigma_N(x,x_m)|
&\leq\tilde{c}_s(N+1)\delta.
 \end{align*}
Using Theorem \ref{offdiag:decay} yields 
 \begin{align*}
 \sum_{x_j \neq x_m}|\sigma_N(x,x_j)|&\leq \frac{C_{0,s}a_{\delta/\nu}}{\nu^s}\leq \frac{c_sa_{\delta/\nu}}{4b},\\
 \sum_{x_j \neq x_m}|X_k^x\sigma_N(x,x_j)|&\leq\frac{C_{1,s}a_{\delta/\nu}(N+1)}{\nu^s}\leq\frac{c_sa_{\delta/\nu}(N+1)}{2b}.
 \end{align*}
 Therfore with Lemma \ref{interpol:solution}u and inserting the paramters
 \begin{align*}
 q(x)&\geq 0.92. 
 \end{align*}
\end{proof}

\begin{Lemma}\label{away:decay}
Under the assumptions of Lemma \ref{defi:Hess}, we have that for all $x \in SO(3)$ with $\omega(x_m^{-1}x)\geq \frac{\pi}{2(N+1)}$ for all $x_m \in \mathcal{C}$ the interpolating function $q$ of Lemma \ref{interpol:solution} with $s=8$ fulfils  
\begin{equation*}
|q(x)| <1.
\end{equation*}
\end{Lemma}
\begin{proof}
 We can estimate
\begin{align}\label{q:est} 
|q(x)| &\leq \|\alpha_{0}\|_{\infty}|\sigma_N(x,x_m)|+\sum_{k=1}^3\|\alpha_k\|_{\infty}|X_k^y\sigma_N(x,x_m)|+\|\alpha_0\|_{\infty}\sum_{x_j \neq x_m}|\sigma_N(x,x_j)| \nonumber\\
 &\quad + \sum_{k=1}^3\|\alpha_k\|_{\infty}\sum_{x_j \neq x_m}|X_k^x\sigma_N(x,x_j)|.
 \end{align}
First assume that there is an $x_m$, such that $\omega(x_m^{-1}x) \leq \frac{2.45\pi}{N+1}$.
Using the Taylor expansion of the cosine function at zero, we have with Lemma \ref{2deriv:lowbound}
\begin{align*}
\sigma_N(x,x_m)&\geq 1-\frac{|\tilde{\sigma}_N^{(2)}(0)|}{2}\omega^2+\frac{|\tilde{\sigma}_N^{(4)}(0)|}{24}\omega^4-\frac{|\tilde{\sigma}_N^{(6)}(0)|}{6!}\omega^6,\\
&\geq 1-\frac{1.001}{36}(N+1)^2\omega^2+\frac{0.999}{24\cdot120}(N+1)^4\omega^4-\frac{1.011}{6!\cdot 11 \cdot 48}(N+1)^6\omega^6.
\end{align*}
It can be shown that the polynomial
\begin{equation*}
1-\frac{1.001}{36}t^2+\frac{0.999}{24\cdot120}t^4-\frac{1.011}{6!\cdot 11 \cdot 48}t^6
\end{equation*}
is positive for $t \in [0,2.45\pi]$ and therefore
\begin{align}\label{tay:upbound}
|\sigma_N(x,x_m)|&=\sigma_N(x,x_m),\nonumber \\
&\leq 1-\frac{|\tilde{\sigma}_N^{(2)}(0)|}{2}\omega^2+\frac{|\tilde{\sigma}_N^{(4)}(0)|}{24}\omega^4,\nonumber \\ 
&\leq1-\frac{1.001}{36}(N+1)^2\omega^2+\frac{0.999}{24\cdot120}(N+1)^4\omega^4.
\end{align}
The r.h.s of \eqref{tay:upbound} is strictly monotonic decreasing for $\omega \in \left[\frac{\pi}{2(N+1)},\frac{t_0}{(N+1)}\right]$ with $t_0=2\sqrt{10\cdot\frac{1.001}{0.999}}$ and strictly increasing for $\omega \in \left[\frac{t_0}{(N+1)},\frac{2.45\pi}{(N+1)}\right]$.

Moreover, we can estimate
\begin{equation}
|X_k^y\sigma_N(x,x_m)| \leq \tilde{c}_s(N+1)t,
\end{equation}
as well as using Lemma \ref{offdiag:decay}
 \begin{align*}
 \sum_{x_j \neq x_m}|\sigma_N(x,x_j)|&\leq \frac{c_sa_{t/\nu}}{4b},\\
 \sum_{x_j \neq x_m}|X_k^x\sigma_N(x,x_j)|&\leq\frac{c_sa_{t/\nu}(N+1)}{2b},
 \end{align*}
 for $t=(N+1)\omega$. Inserting the values $b=28$, $\nu=36$ and $s=8$ for the bounds of the coefficients in Lemma \ref{interpol:solution} results in
\begin{align*}
|q(x)| &\leq 0.96,\quad \text{for } \omega(x_m^{-1}x) \in \left[\frac{\pi}{2(N+1)},\frac{t_0}{(N+1)}\right],\\
|q(x)| &\leq 0.60, \quad \text{for } \omega(x_m^{-1}x) \in \left[\frac{t_0}{(N+1)},\frac{2.45\pi}{(N+1)}\right].
\end{align*}
If there is a $x_m$ such that $\frac{2.45\pi}{N+1} \leq \omega(x_m^{-1}x) \leq \frac{18}{N+1}$, we can estimate in a similar way, but instead of the Taylor expansion we use the asymptotic bounds of Theorem \ref{local:kernelall}, i.e.
\begin{align*}
|\sigma_N(x,x_m)| &\leq \frac{c_{0,8}}{(N+1)^8\omega^8},\\
|X_k^y\sigma_N(x,x_m)| &\leq \frac{c_{1,8}}{(N+1)^7\omega^8}.
\end{align*}
This results in
\begin{equation*}
|q(x)| \leq 0.99
\end{equation*}
for $\frac{2.45\pi}{N+1} \leq \omega(x_m^{-1}x) \leq \frac{18}{N+1}$.
 
 
In the case $\omega(x_j^{-1}x) \geq \frac{18}{(N+1)}$ for all $x_j$, we derive with Lemma \ref{interpol:solution} and estimations similar to those of Lemma \ref{offdiag:decay}
\begin{align*}
|q(x)| &\leq \|\alpha_0\|_{\infty}\sum_{x_j}|\sigma_N(x,x_j)|+\sum_{k=1}^3\|\alpha_k\|_{\infty}\sum_{x_j}|X_k^y\sigma_N(x,x_j)|,\\ 
&\leq \frac{c_sa_{1/2}}{4(b-3)-cs}\leq 0.032.
\end{align*}
\end{proof}

\section{Conclusion and Outlook}\label{sec:conclu}
In this paper, we have shown that under a separation condition on the support of a sum of Dirac measures this measure is the unique solution of the minimization problem \eqref{opti:prob}. Given moments with respect to Wigner D-functions up to order $N$, the sought measure is the unique minimizer if the support set obeys a minimal separation of $\frac{36}{N+1}$, as long as $N\geq 20$. We expect, that this factor is not optimal.
The proof of the uniqueness follows the work of \emph{Cand\'es and Fernandez-Granda} \cite{Candes:2014} and \emph{Bendory et.al.} \cite{Bendory:2015} and seems to be a quite general idea of constructing a dual certificate. Nevertheless, the actual construction heavily depends on precise localization estimates for 'polynomial' interpolation kernels and its derivatives, which have to be verified in a given setting. Moreover we would like to point out, that this is only one way to construct a dual certificate. A different construction, that involves algebraic techniques, was shown by \emph{von der Ohe et.al.} in \cite{vOhe:2016} for Fourier coefficients on the $d-$dimensional Torus and in \cite{vOhe:2016a} for spherical harmonic coefficients on the $d-$dimensional sphere.

This work focussed on the theoretical recovery of the sought measure as the unique solution of the minimization problem \eqref{opti:prob}. The actual computation of the solution, i.e. turning the \emph{infinite}-dimensional minimization \eqref{opti:prob} into a \emph{tractable} problem, is beyond the scope of this paper and will be part of future work.

\appendix
\section{Localized Trigonometric Polynomials}\label{app:loctrig}
In \cite{Mhaskar:2000} and also in \cite{Potts:2007} it was shown that a trigonometric polynomial of the form
\begin{equation*}
\sum_{k=-N}^Ng\left( \frac{k}{2(N+1)}\right)\e^{ikt}
\end{equation*} 
obeys a localization property, as long as the function $g$ is sufficiently smooth with derivatives of \emph{bounded variation}. The variation of a function $f$ defined on $\left[-\frac{1}{2},\frac{1}{2}\right]$ is given by
\begin{equation*}
|f|_V:=\sup\left\{\sum_{i=1}^{n-1}|f(t_{i+1})-f(t_i)|\right\},
\end{equation*}
where the supremum is taken over all partitions $(t_i)_{i=1}^n$ of the interval $[-\frac{1}{2},\frac{1}{2}]$. The space of $(s-1)$-times differentiable functions $g$ with compact support in $\left[-\frac{1}{2},\frac{1}{2}\right]$, such that $|g^{(s-1)}|_V < \infty$, will be denoted as $\mathcal{BV}^{s-1}_0([-\frac{1}{2},\frac{1}{2}])$. Equipped with these definitions we have for $g \in\mathcal{BV}^{s-1}_0([-\frac{1}{2},\frac{1}{2}])$, see \cite[Lemma 3.2]{Potts:2007},
\begin{equation}\label{trigloc:est}
\left| \sum_{k=-N}^Ng\left(\frac{k}{2(N+1)}\right)\e^{ikt} \right| \leq \frac{(2^s-1)\zeta(s)|g^{(s-1)}|_V}{(4(N+1))^{s-1}|t|^{s}},
\end{equation}
for $t \in (0,\pi]$ with $N\geq s-1 \geq 1$. Here $\zeta(s)=\sum_{k=1}^{\infty}\frac{1}{n^s}$ denotes the Riemannian zeta function. 
 
In addition for positive $g \in \mathcal{BV}_0^{s-1}([-\frac{1}{2},\frac{1}{2}])$ one has the bounds, see \cite[Lemma $3.2$]{Potts:2007}.
\begin{equation}\label{l1norm:est}
\left(\|g\|_1-\frac{2\zeta(s)}{(2N\pi )^s}|g^{(s-1)}|_V\right)\leq \frac{\|g\|_{1,N}}{2(N+1)}\leq \left(\|g\|_1+\frac{2\zeta(s)}{(2N\pi )^s}|g^{(s-1)}|_V\right).
\end{equation}
This leads to explicit constants for localization results of the trigonometric polynomial
\begin{equation*}
\tilde{\sigma}_N(t)=\frac{1}{\|g\|_{1,N}}\sum_{k=-N}^Ng\left(\frac{k}{2(N+1)}\right)\e^{ikt},
\end{equation*}
as long as one can compute the $L^1-$norm of the filter function and the total variation of its $(s-1)$-th derivative.

The $l-$th derivative is given by
\begin{equation*}
\tilde{\sigma}_N^{(l)}(t)=\frac{(2(N+1)i)^l}{\|g\|_{1,N}}\sum_{k=-N}^N\left(\frac{k}{2(N+1)}\right)^lg\left(\frac{k}{2(N+1)}\right)\e^{ikt}.
\end{equation*}
Thus, to achieve an analog localization result for the derivatives, we have to estimate the total variation of the function $(z^lg_s(z))^{(s-1)}$. In view of
\begin{equation*}
(z^lg(z))^{(s-1)}=\sum_{n=0}^{l}\binom{s-1}{n}(z^l)^{(n)}g^{(s-1-n)},
\end{equation*}
and
\begin{equation*}
|uv|_{V} \leq \|u\|_{\infty}|v|_{V}+\|v\|_{\infty}|u|_V,
\end{equation*}
for two functions $u$ and $v$ we get for $l \leq s-1$

\begin{align*}
|(z^lg(z))^{(s-1)}|_V
&\leq \sum_{n=0}^l\frac{l!}{(l-n)!}\binom{s-1}{n} \left(\|z^{l-n}\|_{\infty}|g^{(s-1-n)}|_V+\|g^{(s-1-n)}\|_{\infty}|z^{l-n}|_V\right).
\end{align*}
Since $\|z^{l-n}\|_{\infty}=\frac{1}{2^{l-n}}$ and $|z^{l-n}|_V=\frac{1}{2^{l-n-1}}$ for $n< l$ we have
\begin{align}\label{prodvar}
|(z^lg(z))^{(s-1)}|_V&\leq \sum_{n=0}^{l-1}\frac{l!}{(l-n)!}\binom{s-1}{n}\frac{1}{2^{l-n}} \left(|g^{(s-1-n)}|_V+2\|g^{(s-1-n)}\|_{\infty}\right)\\
&\quad+l!\binom{s-1}{l}|g^{(s-1-l)}|_V.\nonumber
\end{align}

To be able to compute the corresponding norms and variations we will choose a specific kernel. We seek for a $(s-1)$-smooth function, whose $(s-1)$-th derivative has small total variation. This leads to functions, whose $(s-1)$-th derivative is piecewise constant. One way to construct such a function is to use a \emph{B-spline} function of order $s-1$, see e.g. \cite{Kunis:2008}. We will use the so called \emph{perfect B-spline} of order $s-1$ as filter function, since in this case $|g^{(s-1)}(x)|=1$.
This function is given by
\begin{equation}\label{perfspline:def}
g_{s-1}(x)=\frac{(-1)^{s-1}}{(s-2)!}\int_{-1}^x\sum_{k=0}^{s-1}(-1)^k\chi_{(\cos(\frac{k+1}{s}\pi),\cos(\frac{k}{s}\pi)]}(t)(x-t)^{s-2}\dx t.
\end{equation}

\begin{proposition}\cite[Sec. $6.1$]{Bojanov:1993}\label{perfspline:prop}
We have for $s \in \N$ that the function $g_{s-1}$ given in \eqref{perfspline:def} is a spline of order $s-1$ with support $[-1,1]$ and $\|g_{s-1}\|_1=\frac{1}{(s-1)!2^{s-2}}$. Moreover we have for all $n \leq s-1$ the identity
\begin{equation*}
g_{s-1}^{(n)}(x)=f_{s-1-n}(x),
\end{equation*}
where
\begin{equation*}
f_0(x)=(-1)^{s-1}\operatorname{sign}(U_{s-1}(x)), \quad f_k(x)=\int_{-1}^xf_{k-1}(t)\dx t,\quad k>0,
\end{equation*}
with the explicit representation for $k>0$,
\begin{equation}\label{explicit:fk}
f_k(x)=\frac{(-1)^{s-1}}{(k-1)!}\int_{-1}^x(x-t)^{k-1}\operatorname{sign}(U_{s-1}(t)) \dx t. 
\end{equation}
Here $U_{s-1}$ denotes the Chebychev polynomial of the second kind of order $s-1$.
\end{proposition}

The scaled function
\begin{equation*}
\tilde{g}_{s-1}(x)=g_{s-1}(2x).
\end{equation*}
has its support in $[-\frac{1}{2},\frac{1}{2}]$ and we have $\tilde{g}_{s-1} \in \mathcal{BV}_0^{s-1}([-\frac{1}{2},\frac{1}{2}])$. We define the kernel by
\begin{equation*}
\tilde{\sigma}_N(t)=\frac{1}{\|\tilde{g}_s\|_{1,N}}\sum_{k=-N}^N\tilde{g}_{s-1}\left(\frac{k}{2(N+1)}\right)\e^{ikt}.
\end{equation*}

\begin{Lemma}\label{var:const}
For $s \in 2\N, s \geq 6$  we have
\begin{align*}
&|\tilde{g}_{s-1}^{(s-1)}|_{V}=2^{s}s,\quad \|\tilde{g}_{s-1}^{(s-1)}\|_{\infty}=2^{s-1}\\
&|\tilde{g}_{s-1}^{(s-2)}|_{V}=2^{s-1},\quad \|\tilde{g}_{s-1}^{(s-2)}\|_{\infty}=2^{s-2}\tan\left(\frac{\pi}{2s}\right)\\
&|\tilde{g}_{s-1}^{(s-3)}|_{V}=2^{s-4}\tan^2\left(\frac{\pi}{2s}\right)s, \quad \|\tilde{g}_{s-1}^{(s-3)}\|_{\infty} = 2^{s-4}\tan^2\left(\frac{\pi}{2s}\right),\\
&|\tilde{g}_{s-1}^{(s-4)}|_{V}\leq2^{s-4}\tan^2\left(\frac{\pi}{2s}\right), \quad \|\tilde{g}_{s-1}^{(s-4)}\|_{\infty}=2^{s-4}\frac{3\sin^2\left(\frac{\pi}{2s}\right)\tan\left(\frac{\pi}{2s}\right)}{2\cos\left(\frac{\pi}{s}\right)-1},\\
&|\tilde{g}_{s-1}^{(s-5)}|_{V}\leq2^{s-4}\frac{3\sin^2\left(\frac{\pi}{2s}\right)\tan\left(\frac{\pi}{2s}\right)}{2\cos\left(\frac{\pi}{s}\right)-1}.
\end{align*}
In the case $s=8$ we have in addition
 \begin{align*}
 \|\tilde{g}_{7}^{(j)}\|_{\infty}\leq \frac{4^j}{2^5(6-j)!},\quad |\tilde{g}_{7}^{(j-1)}|_{V}\leq\frac{4^j}{2^4(6-j)!}, \quad j=1,2,3. 
%
%
\end{align*}
\end{Lemma}

\begin{proof}
First observe that we have for $0\leq n \leq 3$
\begin{equation*}
|\tilde{g}_{s-1}^{(s-1-n)}|_V=2^{(s-1-n)}|g_{s-1}^{(s-1-n)}|_V, \quad \|\tilde{g}_{s-1}^{(s-1-n)}\|_{\infty}=2^{(s-1-n)}\|g_{s-1}^{(s-1-n)}\|_{\infty}.
\end{equation*}
By Proposition \ref{perfspline:prop} we have that $g_{s-1}^{(s-1-n)}=f_n$. For $n=0$ this means that
\begin{equation*}
g_{s-1}^{(s-1)}=f_0(x)=(-1)^{s-1}\operatorname{sign}(U_{s-1}(x))
\end{equation*}
and since $U_{s-1}(x)$ has $s-1$ zeros in the interval $(-1,1)$ and is not equal to zero for $x=1, -1$ we have that
\begin{equation}
|g_{s-1}^{(s-1)}|=|f_0|_V=2s, \quad \|g_{s-1}^{(s-1)}\|_{\infty}=\|f_0\|_{\infty}=1. 
\end{equation}
For $n=1$ The total variation of $f_1$ is given by
\begin{equation}
|f_1|_V=\sup\left(\sum_i|f_1(x_{i+1})-f_1(x_i)|\right)=\sup\left(\sum_i\left|\int_{x_i}^{x_{i+1}}f_0(t)\dx t\right|\right) \leq 2,
\end{equation}
where the supremum is taken over all partitions of $[-1,1]$. Actually choosing as partition the sequence of zeros of $U_s$ shows that $|g_{s-1}^{(s-2)}|_V=2$.
Moreover we have the representation
\begin{align}\label{f1:pwl}
f_1(x)&=\sum_{k=0}^{s-1}(-1)^k\chi_{\left(\cos\left(\frac{(s-k)\pi}{s}\right),\cos\left(\frac{(s-k-1)\pi}{s}\right)\right]}(x)\\
&\quad \cdot \left(x-\cos\left(\frac{(s-k)\pi}{s}\right)-\tan\left(\frac{\pi}{2s}\right)\sin\left(\frac{\pi k}{s}\right)\right),\nonumber
\end{align}
which shows
\begin{equation*}
\|g_{s-1}^{(s-2)}\|_{\infty}=\|f_1\|_{\infty}=|f_1(0)|=\tan\left(\frac{\pi}{2s}\right).
\end{equation*}

Since $f_2$ is continuously differentiable with $f'_2=f_1$ we have 
\begin{equation*}
|f_2|_V=\int_{-1}^1|f'_2(t)| \dx t= \int_{-1}^1|f_1(t)| \dx t.
\end{equation*}
Using the representation \eqref{f1:pwl} of $f_1$ a lengthy calculation shows
\begin{equation*}
|g_{s-1}^{(s-3)}|_{V}=|f_2|_V=\frac{s}{2}\tan^2\left(\frac{\pi}{2s}\right).
\end{equation*}

For $n=2$ we have the representation
 \begin{align}\label{f2:pwq}
f_2(x)&=\sum_{k=0}^{s-1}(-1)^k\chi_{\left(\cos\left(\frac{(s-k)\pi}{s}\right),\cos\left(\frac{(s-k-1)\pi}{s}\right)\right]}(x)\\ \nonumber
&\quad \cdot \frac{1}{2}\Bigg(x^2+2\left(\cos\left(\frac{k\pi}{s}\right)-\tan\left(\frac{\pi}{2s}\right)\sin\left(\frac{\pi k}{s}\right)\right)x\\
&\quad+\frac{1}{2}\left(1+\cos\left(\frac{2\pi k}{s}\right)-\sin\left(\frac{2\pi k}{s}\right)\tan\left(\frac{\pi}{s}\right)\right) \Bigg).\nonumber
\end{align}
Since the local extrema of $f_2$ are the zeros of $f_1$, which are given by
\begin{equation*}
\tan\left(\frac{\pi}{2s}\right)\sin\left(\frac{\pi k}{s}\right)-\cos\left(\frac{k\pi}{s}\right).
\end{equation*}
A lengthy calculation shows that the absolute values at these points are given by
\begin{equation*}
\frac{1}{4}\tan^2\left(\frac{\pi}{2s}\right)\left(1-\frac{\cos\left(\frac{2k+1}{s}\pi\right)}{\cos\left(\frac{\pi}{s}\right)}\right)
\end{equation*}
which becomes maximal for $k=\frac{s}{2}-1$. Since
\begin{equation*}
|f_3|_V=\int_{-1}^1|f'_3(t)| \dx t= \int_{-1}^1|f_2(t)| \dx t,
\end{equation*}
we get immediately $|f_3|_V\leq \tan^2\left(\frac{\pi}{2s}\right)$. In the same way we can derive a piecewise representation of $f_3$
and plug in the zeros of $f_2$. A lengthy calculation for this shows 
\begin{equation*}
\|g_{s-1}^{(s-4)}\|_{\infty}=\|f_3\|_{\infty}=|f(0)|=\frac{1}{24}\left(\tan\left(\frac{3\pi}{2s}\right)-3\tan\left(\frac{\pi}{2s}\right)\right)=\frac{3\sin^2\left(\frac{\pi}{2s}\right)\tan\left(\frac{\pi}{2s}\right)}{2\cos\left(\frac{\pi}{s}\right)-1},
\end{equation*}
and again $|g_{s-1}^{(s-5)}|_{V}\leq 2 \|f_3\|_{\infty}$.

For the case $s=8$ we use the bound
\begin{equation*}
 \|g_{s-1}^{(j)}\|_{\infty} \leq \frac{2^{j+1}}{2^{s-2}(s-j-2)!},
\end{equation*} 
see \cite[Thm. $4.36$]{Schumaker:2007} with a different normalization of the spline and again $|g_{s-1}^{(j-1)}|_V\leq 2 \|g_{s-1}^{(j)}\|_{\infty}$.


\end{proof}

\begin{Theorem}\label{trig:est}
Let $s \geq 6$, $s\in 2 \N$ and $N \geq 2s$. Using the scaled perfect B-spline $\tilde{g}_{s-1}$ as filter function leads to localization estimates for the kernel $\tilde{\sigma}_N$ and it's derivatives up to order $3$, i.e. for $t \in [-\pi,\pi]\setminus\{0\}$ we have 
\begin{equation}
|\tilde{\sigma}_N^{(l)}(t)| \leq \frac{c_{l,s}}{(N+1)^{s-l}|t|^s}, \quad l=0,\dots,3,
\end{equation} 
where the constants are given by
\begin{align*}
c_{0,s}&= 1.02 \cdot (s-1)! \cdot 2^s \cdot s ,\\
c_{1,s}&=1.02\cdot (s-1)! \cdot 2^{s} \cdot 2s,\\
c_{2,s}&=1.02\cdot (s-1)! \cdot2^{s}\cdot (4s+1),\\
c_{3,s}&=1.02\cdot (s-1)!\cdot 2^{s}\cdot\left(9s-2\right).
\end{align*}
\end{Theorem}

\begin{proof}
The kernel and it's derivatives up to order $3$ are given by
\begin{equation*}
\tilde{\sigma}_N^{(l)}(t)=\frac{(2(N+1)i)^l}{\|\tilde{g}_{s-1}\|_{1,N}}\sum_{k=-N}^N\left(\frac{k}{2(N+1)}\right)^l\tilde{g}_{s-1}\left(\frac{k}{2(N+1)}\right)\e^{ikt}, \quad l=0,\dots,3.
\end{equation*}
In view of Proposition \ref{perfspline:prop} and Lemma \ref{var:const} $\|\tilde{g}_{s-1}\|_1=\frac{1}{2}\|g_{s-1}\|_{1}=\frac{1}{(s-1)!2^{s-1}}$ and $|\tilde{g}_{s-1}^{(s-1)}|_{V}=2^{s}s$. Thus by the estimation \eqref{l1norm:est} we can bound the discrete norm $\|\tilde{g}_{s-1}\|_{1,N}$ from below by
\begin{equation}
\|\tilde{g}_{s-1}\|_{1,N}\geq \frac{N+1}{2^{s-2}}\left(\frac{1}{(s-1)!}-\frac{\zeta(s)s}{(\pi s)^s}\right). 
\end{equation}

Combining this with the localization estimate \eqref{trigloc:est} yields
\begin{align*}
|\tilde{\sigma}_N^{(l)}(t)|&\leq \frac{(2(N+1))^l(2^s-1)\zeta(s)}{4^{s-1}(N+1)^{s-1}|t|^s}|(z^l\tilde{g}_{s-1})^{(s-1)}|_V\frac{2^s}{4(N+1)\left(\frac{1}{(s-1)!}-\frac{\zeta(s)s}{(\pi s)^s}\right)},\\
&= \frac{2^l|(z^l\tilde{g}_{s-1})^{(s-1)}|_V}{(N+1)^{s-l}|t|^s}\frac{(2^s-1)\zeta(s)}{2^s\left(\frac{1}{(s-1)!}-\frac{\zeta(s)s}{(\pi s)^s}\right)},\\
&=\frac{2^l|(z^l\tilde{g}_{s-1})^{(s-1)}|_V}{(N+1)^{s-l}|t|^s}\frac{(s-1)!}{\left(\frac{1}{\zeta(s)}-\frac{s(s-1)!}{(\pi s)^s}\right)}.
\end{align*}

Now observe that the sequences $\frac{1}{\zeta(s)}$ and $\frac{-s(s-1)!}{(\pi s)^s}$ are monotonically increasing in $s$. This means we can bound them from below by the first possible value for $s$, that  is $s=6$. This gives the upper bound
\begin{equation*}
\frac{1}{\left(\frac{1}{\zeta(s)}-\frac{s(s-1)!}{(\pi s)^s}\right)} \leq \frac{1}{\left(\frac{945}{\pi^6}-\frac{720}{\pi^6 6^6}\right)} \leq 1.02,
\end{equation*}
which results in
\begin{equation*}
|\tilde{\sigma}_N^{(l)}(t)| \leq \frac{1.02(s-1)!2^l|(z^l\tilde{g}_{s-1})^{(s-1)}|_V}{(N+1)^{s-l}|t|^s}.
\end{equation*}

Using the inequality \eqref{prodvar} for the variation of the derivative of the products together with Lemma \ref{var:const} and the estimation $\tan\left(\frac{\pi}{2s}\right)\leq \frac{2}{s}$ for $s\geq 6$ yields
\begin{align*}
|(z\tilde{g}_{s-1}(z))^{(s-1)}|_V&\leq 2^{s}s ,\\
|(z^2\tilde{g}_{s-1}(z))^{(s-1)}|_V&\leq 2^{s-2}\left(4s+1\right),\\
|(z^3\tilde{g}_{s-1}(z))^{(s-1)}|_V&\leq 2^{s-3}\left(9s-2\right).
\end{align*}
 and therefore the constants.
\end{proof}

In view of the constants appearing in the previous Theorem, the bounds are only meaningful if $\omega$ is well separated from zero. For values of $\omega$ close to zero we will use different bounds derived from series expansion around zero. For this purpose we need upper and lower bounds for the values of the second and fourth derivative of $\tilde{\sigma}_N$ at zero.

\begin{Lemma}\label{2deriv:lowbound}
Let $s \geq 8$, $s\in 2 \N$ and $N \geq 2s$. Using the scaled perfect B-spline $\tilde{g}_{s-1}$ as filter function leads to the following bounds
\begin{equation*}
c_s(N+1)^2 \leq |\tilde{\sigma}_N^{''}(0)| \leq \tilde{c}_s(N+1)^2, 
\end{equation*}
with
\begin{equation*}
 c_s=\frac{0.999}{2(s+1)},\quad \tilde{c}_s=\frac{1.001}{2(s+1)},
\end{equation*}
and
\begin{equation*}
 d_s(N+1)^4\leq |\tilde{\sigma}_N^{(4)}(0)| \leq \tilde{d}_s(N+1)^4,
\end{equation*}
with
\begin{equation*}
d_s=\frac{3\cdot 0.999}{4(s+2)(s+1)}, \quad \tilde{d}_s= \frac{3\cdot1.001}{4(s+2)(s+1)}.
\end{equation*}
In the case $s=8$ we have for $N \geq 20$
\begin{equation}
|\tilde{\sigma}_N^{(6)}(0)|\leq 1.011\cdot \frac{15}{8}\cdot \frac{8!}{11!}\cdot(N+1)^6.
\end{equation}
\end{Lemma}

\begin{proof}
We have for $m \in \N$
\begin{align*}
|\tilde{\sigma}_N^{(2m)}(0)|&=\frac{(2(N+1))^{2m}}{\|\tilde{g}_{s-1}\|_{1,N}}\sum_{k=-N}^N\left(\frac{k}{2(N+1)}\right)^{2m}\tilde{g}_{s-1}\left(\frac{k}{2(N+1)}\right).
\end{align*}
To estimate the expressions
\begin{align*}
\|z^{2m}\tilde{g}_{s-1}(z)\|_{1,N}&=\sum_{k=-N}^N\left(\frac{k}{2(N+1)}\right)^{2m}\tilde{g}_{s-1}\left(\frac{k}{2(N+1)}\right),\\
\end{align*}
using inequality \eqref{l1norm:est} we have to calculate the $L^1$-norm of the functiosn $z^{2m}\tilde{g}_{s-1}$ on $[-\frac{1}{2},\frac{1}{2}]$. By Proposition \ref{perfspline:prop}
we have $g_{s-1}(x)=f_{s-1}(x)$ and integration by parts shows
\begin{align*}
\int_{-\frac{1}{2}}^{\frac{1}{2}}x^{2m}\tilde{g}_{s-1}(x) \dx x&=\frac{1}{2^{2m+1}}\sum_{l=0}^{2m}(-1)^l\frac{(2m)!}{(2m-l)!}f_{s+l}(1).
\end{align*} 
Now we use a specific orthogonality relation for Chebeychev polynomials of the second type, see e.g. \cite{Bojanov:1993}. For each polynomial $p$ of maximal degree $s-2$ we have
\begin{equation*}
\int_{-1}^1p(t)\operatorname{sign}(U_{s-1}(t)) \dx t=0,
\end{equation*}
and for $m \in \N$
\begin{equation*}
\int_{-1}^{1}t^{s+2m}\operatorname{sign}(U_{s-1}(t))\dx t=0
\end{equation*}
since the functions $t^{s+2m}\operatorname{sign}(U_{s-1}(t))$ are always odd.
Moreover we can calculate
\begin{align*}
\int_{-1}^{1}t^{s+2m-1}\operatorname{sign}(U_{s-1}(t))\dx t&=\int_{-1}^1t^{s+2m-1}\sum_{k=0}^{s-1}(-1)^k\chi_{(\cos(\frac{k+1}{s}\pi),\cos(\frac{k+1}{s}\pi)]}(t) \dx t,\\
&=\frac{2}{s+2m}\left(1+\sum_{k=1}^{s-1}(-1)^k\cos^{s+2m}\left(\frac{k}{s}\pi\right)\right).
\end{align*}
Since 
\begin{equation*}
\sum_{k=1}^{s-1}(-1)^k\cos^{s+2}\left(\frac{k}{s}\pi\right)=\frac{s}{2^{s+2m-1}}\binom{s+2m}{m}-1
\end{equation*}
we have
\begin{align*}
\int_{-1}^1t^{s+2m-1}\operatorname{sign}(U_{s-1}(t))\dx t= \frac{s\cdot (s+2m-1)!}{2^{s+2m-2}\cdot m! \cdot(s+m)!}.
\end{align*}
Thus, using the explicit representation \eqref{explicit:fk} of the $f_k$ yields
\begin{align*}
f_{s+l}(1)&=\frac{1}{(s-1)!2^{s-2}}\sum_{r=0}^{\frac{l}{2}}\frac{1}{(l-2r)!}\frac{s!}{4^r\cdot r! \cdot (s+r)!},\quad l \text{ even },\\
f_{s+l}(1)&=\frac{1}{(s-1)!2^{s-2}}\sum_{r=0}^{\frac{l-1}{2}}\frac{1}{(l-2r)!}\frac{s!}{4^r\cdot r! \cdot (s+r)!},\quad l \text{ odd }.
\end{align*}
Using this we can derive
\begin{align*}
\|z^{2m}\tilde{g}_{s-1}(z)\|_{1}&=\frac{(2m)! \cdot s}{4^m\cdot m! \cdot 2^{s+2m-1}(s+m)!}.
\end{align*}

Together we get by applying inequality \eqref{l1norm:est}
\begin{align}\label{qout:est}
|\tilde{\sigma}_N^{''}(0)|&=(2(N+1))^2\frac{\|z^2\tilde{g}_{s-1}(z)\|_{1,N}}{\|\tilde{g}_{s-1}\|_{1,N}},\nonumber\\
&\geq (2(N+1))^2\frac{\left(\|z^2\tilde{g}_{s-1}(z)\|_{1}-\frac{2\zeta(s)}{(4\pi s)^s}|(z^2\tilde{g}_{s-1}(z))^{(s-1)}|_V\right)}{\left(\|\tilde{g}_{s-1}\|_1+\frac{2\zeta(s)}{(4\pi s)^s}|\tilde{g}_{s-1}^{(s-1)}|_V\right)},\nonumber\\
&\geq \frac{(2(N+1))^2}{8(s+1)} \frac{1-2(s+1)(s-1)!\frac{\zeta(s)}{(\pi s)^s}\left(4s+1\right)}{1+s!\frac{\zeta(s)}{(\pi s)^s}}.
\end{align}
We can bound the second quotient in \eqref{qout:est} from below by its value at $s=8$, i.e
\begin{align*}
 \frac{1-2(s+1)(s-1)!\frac{\zeta(s)}{(\pi s)^s}\left(4s+1\right)}{1+s!\frac{\zeta(s)}{(\pi s)^s}} \geq \frac{1-\frac{25}{4\cdot3^8}}{1+\frac{1}{35\cdot 3^8\cdot 2^5}}\geq 0.999.
\end{align*}
Using the same argument we can bound from above
\begin{align*}
|\tilde{\sigma}_N^{''}(0)|&\leq\frac{(2(N+1))^2}{8(s+1)} \frac{1+2(s+1)(s-1)!\frac{\zeta(s)}{(\pi s)^s}\left(4s+1\right)}{1-s!\frac{\zeta(s)}{(\pi s)^s}},\\
&\leq \frac{1.001}{2(s+1)}(N+1)^2.
\end{align*}
By Lemma \ref{var:const} and the inequality \eqref{prodvar} as well as $\sin\left(\frac{\pi}{2s}\right)\leq \frac{\pi}{2s}$ and $\cos\left(\frac{\pi}{2s}\right)\left(2\cos\left(\frac{\pi}{s}\right)-1\right) \geq \frac{3}{4}$ we have
\begin{equation*}
|(z^4\tilde{g}_{s-1}(z))^{(s-1)}|_V\leq 2^{s-4}(35s-19)
\end{equation*}
and therefore with the same argumentation as before
\begin{align*}
|\tilde{\sigma}_N^{(4)}(0)| &\leq \frac{3(N+1)^4}{4(s+2)(s+1)}\frac{1+\frac{4}{3}\frac{\zeta(s)(s+2)!}{(\pi s)^ss}(35s-19)}{1-\frac{\zeta(s)s!}{(\pi s)^s}},\\
&\leq \frac{3(N+1)^4}{4(s+2)(s+1)} 1.001,
\end{align*}
and
\begin{equation*}
|\tilde{\sigma}_N^{(4)}(0)| \geq \frac{3(N+1)^4}{4(s+2)(s+1)} 0.999.
\end{equation*}

In the case $s=8$ we have again by using Lemma \ref{var:const} and the inequality \eqref{prodvar}
\begin{equation*}
|(z^6\tilde{g}_{7}(z))^{(7)}|_V\leq 5.0896 \cdot 10^{3}.
\end{equation*}
Using the same argument as before shows
\begin{equation*}
|\tilde{\sigma}_N^{(6)}(0)|\leq 1.011\cdot \frac{15}{8}\cdot \frac{s!}{(s+3)!}\cdot(N+1)^6=1.011\cdot \frac{15}{8}\cdot \frac{8!}{11!}\cdot(N+1)^6.
\end{equation*}   
\end{proof}




\section{Proofs}\label{app:proofs}

\subsection{Proof of Lemma \ref{local:kernelall2}}
\begin{proof}
The proof works in the same way as the proof of Theorem \ref{local:kernelall}. First the derivatives are calculated directly via the product rule, then the according terms are grouped together in the right way, and an estimation is shown.
For abbreviation we suppress the dependence on $x,y$ in the following. By the product rule we have
\begin{equation}
X_i^xX_i^xX_k^y\sigma_N=-X_i^xX_i^x e_k\tilde{\sigma}_N^{'}-(2e_iX_i^xe_k+e_kX_i^xe_i)\tilde{\sigma}_N^{''}-e_i^2e_k\tilde{\sigma}_N^{'''}.
\end{equation}

Suppose we have
\begin{equation*}
X_i^xe_k=-e_ie_k\left(\frac{1+\cos(\omega)}{2\sin(\omega)}\right)\pm \frac{e_j}{2},
\end{equation*}
see \eqref{deriv:rotax}, then the factor in front of $\tilde{\sigma}_N^{'}$ is calculated as
\begin{align*}
X_i^xX_i^xe_k&=\left(\frac{1+\cos(\omega)}{2\sin(\omega)}\right)^2\left(e_i^2e_k-(1-e_i^2)e_k\right)+\left(\frac{1+\cos(\omega)}{2\sin(\omega)}\right)\left(\frac{e_i^2e_k}{\sin(\omega)}\mp  e_ie_j\right)\\
&\quad \mp\frac{e_k}{4},
\end{align*}
and the factor in front of $\tilde{\sigma}_N^{''}$ as
\begin{align*}
(2e_iX_i^xe_k+e_kX_i^xe_i)&=\left(\frac{1+\cos(\omega)}{2\sin(\omega)}\right)\left(-2e_i^2e_k+e_k(1-e_i^2)\right)\pm e_ie_j.
\end{align*}

This means
\begin{align}\label{thirdderiv:1}
X_i^xX_i^xX_k^y\sigma_N&=\left((3e_ke_i^2-e_k)\left(\frac{1+\cos(\omega)}{2\sin(\omega)}\right)\mp e_ie_j\right)\tilde{\sigma}_N^{''}(\omega)\nonumber \\
&\quad + \Bigg((2e_ke_i^2-e_k)\left(\frac{1+\cos(\omega)}{2\sin(\omega)}\right)^2-e_ke_i^2\left(\frac{1+\cos(\omega)}{2\sin^2(\omega)}\right) \nonumber \\
&\quad \pm e_ie_j\left(\frac{1+\cos(\omega)}{2\sin(\omega)}\right)\pm \frac{e_k}{4}\Bigg)\tilde{\sigma}_N^{'}(\omega)\nonumber \\
&\quad -e_i^2e_k\tilde{\sigma}_N^{'''}.
\end{align}
Therefore, by again using \eqref{sin:lowest}, we have for $\omega \in (\frac{\pi}{2(N+1)},\frac{\pi}{2}]$
\begin{align*}
|X_i^xX_i^xX_k^y\sigma_N|&\leq\left(|3e_ke_i^2-e_k|(N+1)+|e_ie_j|\right)|\tilde{\sigma}_N^{''}|+|e_i^2e_k||\tilde{\sigma}_N^{'''}|\\
&\quad +\Bigg(\left|e_ke_i^2\cos(\omega)-e_k\left(\frac{1+\cos(\omega)}{2}\right)\right|(N+1)^2\\
&\quad+|e_ie_j|(N+1)+\frac{|e_k|}{4}\Bigg)|\tilde{\sigma}_N^{'}|.
\end{align*}
Now we use that $|3e_ke_i^2-e_k|\leq 1$, $|e_ke_i^2|, |e_ie_k^2| \leq \frac{2}{3\sqrt{3}}$, $|e_ie_j|\leq \frac{1}{2}$, and
\begin{equation*}
\left|e_ke_i^2\cos(\omega)-e_k\left(\frac{1+\cos(\omega)}{2}\right)\right| \leq 1,
\end{equation*}
to derive for $N\geq 2s \geq 12$ 
\begin{align*}
|X_i^xX_i^xX_k^y\sigma_N|
&\leq (N+1) 1.04 |\tilde{\sigma}_N^{''}|+\frac{2}{3\sqrt{3}}|\tilde{\sigma}_N^{'''}|+(N+1)^21.04|\tilde{\sigma}_N^{'}|.
\end{align*}
With Theorem \ref{trig:est} and the observation that
\begin{align*}
\frac{c_{2,s}}{c_{3,s}}&=\frac{4s+1}{9s-2}\leq \frac{1}{2},\\
\frac{c_{1,s}}{c_{3,s}}&=\frac{2s}{9s-2}\leq \frac{1}{4},
\end{align*}
we have
\begin{equation*}
|X_i^xX_i^xX_k^y\sigma_N(x,y)| \leq \frac{1.2\cdot c_{3,s}}{(N+1)^{s-3}\omega(y^{-1}x)^s}.
\end{equation*}
In the case $k=i$, we calculate
\begin{align}\label{thirdderiv:2}
X_i^xX_i^xX_i^y\sigma_N&=-3e_i(1-e_i^2) \left(\frac{1+\cos(\omega)}{2\sin(\omega)}\right)\tilde{\sigma}_N^{''}(\omega)+2e_i(1-e_i^2) \left(\frac{1+\cos(\omega)}{2\sin(\omega)}\right)^2\tilde{\sigma}_N^{'}(\omega)\nonumber \\
&\quad+e_i(1-e_i^2)\left(\frac{1+\cos(\omega)}{2\sin^2(\omega)}\right)\tilde{\sigma}_N^{'}(\omega)-e_i^3\tilde{\sigma}_N^{'''}(\omega),
\end{align}
which yields
\begin{align*}
|X_i^xX_i^xX_i^y\sigma_N| &\leq 3|e_i(1-e_i^2)|(N+1)\left(|\tilde{\sigma}_N^{''}|+(N+1)^2|\tilde{\sigma}_N^{'}|\right)+|e_i^3||\tilde{\sigma}_N^{'''}|,\\
&\leq \frac{1}{(N+1)^{s-3}\omega^s}|e_i|\left(3(1-e_i^2)\left(c_{2,s}+c_{1,s}\right)+e_i^2c_{3,s}\right),\\
&\leq \frac{c_{3,s}}{(N+1)^{s-3}\omega^s}\underbrace{|e_i|(2.25-1.25e_i^2)}_{\leq 1.2} \leq \frac{1.2 \cdot c_{3,s}}{(N+1)^{s-3}\omega^s}.
\end{align*}
This shows the estimate for the on-diagonal entries of the Hessian. 
For the off-diagonal entries observe, that we have for $n \neq j,i$ the following sign combination
\begin{align*}
X_i^xe_n&=-e_ie_n\left(\frac{1+\cos(\omega)}{2\sin(\omega)}\right)\pm \frac{e_j}{2},\\
X_j^xe_n&=-e_je_n\left(\frac{1+\cos(\omega)}{2\sin(\omega)}\right)\mp \frac{e_i}{2},\\
X_j^xe_i&=-e_je_i\left(\frac{1+\cos(\omega)}{2\sin(\omega)}\right)\pm \frac{e_n}{2}.
\end{align*}
This gives
\begin{align*}
 X_i^xX_n^x\sigma_N \mp \frac{1}{2}X_j^x\sigma_N=e_ie_n\left(\tilde{\sigma}_N^{''}(\omega)-\frac{1+\cos(\omega)}{2\sin(\omega)}\tilde{\sigma}_N^{'}(\omega)\right)
\end{align*} 
and therefore, using \eqref{sin:lowest} and Theorem \ref{trig:est},
\begin{align*}
|X_i^xX_n^x\sigma_N \mp \frac{1}{2}X_j^x\sigma_N|&\leq |e_ie_n|\left(|\tilde{\sigma}_N^{''}(\omega)|+(N+1)|\tilde{\sigma}_N^{'}(\omega)|\right),\\
&\leq \frac{|e_ie_n|}{(N+1)^{s-2}\omega^s}\left(c_{2,s}+c_{1,s}\right).
\end{align*}
Since $\frac{c_{1,s}}{c_{2,s}} \leq \frac{1}{2}$ and $|e_ie_n| \leq \frac{1}{2}$,
 we have
\begin{equation*}
|X_i^xX_n^x\sigma_N\mp \frac{1}{2}X_j^x\sigma_N|\leq  \frac{c_{2,s}}{(N+1)^{s-2}\omega^s},
\end{equation*}
which shows the first inequality. For the second one we calculate
\begin{align*}
X_j^xX_i^xX_n^y\sigma_N \mp \frac{1}{2} X_n^xX_n^y\sigma_N&=-X_j^xX_i^xe_n\tilde{\sigma}_N'(\omega)-\left(e_jX_i^xe_n+e_nX_j^xe_i+e_iX_j^xe_n\right)\tilde{\sigma}_N^{''}(\omega)\\
&\quad -e_ie_je_n\tilde{\sigma}_N^{'''}(\omega)\mp \frac{1}{2}\left(-X_n^xe_n\tilde{\sigma}_N'(\omega)-e_n^2\tilde{\sigma}_N^{''}(\omega)\right),\\
&=-X_j^xX_i^xe_n\tilde{\sigma}_N'(\omega)-\left(e_jX_i^xe_n+e_nX_j^xe_i+e_iX_j^xe_n\right)\tilde{\sigma}_N^{''}(\omega)\\
&\quad-e_ie_je_n\tilde{\sigma}_N^{'''}(\omega)\pm \frac{1}{2}\left((1-e_n^2)\left(\frac{1+\cos(\omega)}{2\sin(\omega)}\right)\tilde{\sigma}_N'(\omega)+e_n^2\tilde{\sigma}_N^{''}(\omega)\right),
\end{align*}
with
\begin{align*}
X_j^xX_i^xe_n&=2e_je_ie_n\left(\frac{1+\cos(\omega)}{2\sin(\omega)}\right)^2\mp \frac{e_n^2+e_j^2-e_i^2}{2}\left(\frac{1+\cos(\omega)}{2\sin(\omega)}\right)\\
&\quad +\frac{e_ie_je_n}{\sin(\omega)}\left(\frac{1+\cos(\omega)}{2\sin(\omega)}\right)\pm\frac{1}{2}\left(\frac{1+\cos(\omega)}{2\sin(\omega)}\right)
\end{align*}
and
\begin{align*}
\left(e_jX_i^xe_n+e_nX_j^xe_i+e_iX_j^xe_n\right)&=(-3e_ie_ne_j)\left(\frac{1+\cos(\omega)}{2\sin(\omega)}\right)\pm\frac{e_j^2+e_n^2-e_i^2}{2}.
\end{align*}
Putting this together yields 
\begin{align}\label{thirdderiv:3}
X_j^xX_i^xX_n^y\sigma_N \mp \frac{1}{2} X_n^xX_n^y\sigma_N&=\left(3e_ie_je_n\left(\frac{1+\cos(\omega)}{2\sin(\omega)}\right)\pm \frac{e_i^2-e_j^2}{2}\right)\tilde{\sigma}_N^{''}(\omega) \nonumber \\
&\quad -\Bigg(2e_ie_je_n\left(\frac{1+\cos(\omega)}{2\sin(\omega)}\right)^2+e_ie_je_n\left(\frac{1+\cos(\omega)}{2\sin^2(\omega)}\right) \nonumber \\
&\quad \pm\frac{e_i^2-e_j^2}{2}\left(\frac{1+\cos(\omega)}{2\sin(\omega)}\right)\Bigg)\tilde{\sigma}_N^{'}(\omega)- e_ie_ne_j\tilde{\sigma}_N^{'''}(\omega).
\end{align}
Again using \eqref{sin:lowest}, we can estimate
\begin{align*}
|X_j^xX_i^xX_n^y\sigma_N\mp \frac{1}{2} X_n^xX_n^y\sigma_N| &\leq \left(3|e_ie_je_n|(N+1)+ \frac{|e_i^2-e_j^2|}{2}\right)|\tilde{\sigma}_N^{''}(\omega)| \nonumber \\
&\quad +\left(3|e_ie_je_n|(N+1)^2+\frac{|e_i^2-e_j^2|}{2}(N+1)\right)|\tilde{\sigma}_N^{'}(\omega)|\\
&\quad + |e_ie_ne_j||\tilde{\sigma}_N^{'''}(\omega)|.
\end{align*}
With Theorem \ref{trig:est} and $|e_ie_ne_j| \leq \left(\frac{1}{\sqrt{3}}\right)^3 \leq \frac{1}{5}$, as well as $\frac{c_{2,s}}{c_{3,s}}\leq \frac{1}{2}$ and $\frac{c_{1,s}}{c_{3,s}}\leq \frac{1}{4}$, we have for $N\geq 2s \geq 12$
\begin{align*}
|X_j^xX_i^xX_n^y\sigma_N\mp \frac{1}{2} X_n^xX_n^y\sigma_N| 
&\leq  \frac{c_{3,s}}{(N+1)^{s-3}\omega^s}.
\end{align*}
In the same way one calculates 
\begin{align}\label{thirdderiv:4}
X_j^xX_i^xX_i^y\sigma_N \mp \frac{1}{2} X_n^xX_i^y\sigma_N&=\left(e_j(3e_i^2-1)\left(\frac{1+\cos(\omega)}{2\sin(\omega)}\right)\mp\frac{e_ie_j}{2} \right)\tilde{\sigma}_N^{''}(\omega) \nonumber \\
&\quad -\Bigg(2e_i^2e_j\left(\frac{1+\cos(\omega)}{2\sin(\omega)}\right)^2+e_j(1-e_i^2)\left(\frac{1+\cos(\omega)}{2\sin^2(\omega)}\right) \nonumber \\
&\quad \mp \frac{e_ie_n}{2}\left(\frac{1+\cos(\omega)}{2\sin(\omega)}\right)\pm\frac{e_j}{4}\Bigg)\tilde{\sigma}_N^{'}(\omega)- e_i^2e_j\tilde{\sigma}_N^{'''}(\omega),
\end{align}
leading to the estimation
\begin{align*}
|X_j^xX_i^xX_i^y\sigma_N \mp \frac{1}{2} X_n^xX_i^y\sigma_N|&\leq\left(|e_j(3e_i^2-1)|(N+1)+\frac{|e_ie_j|}{2} \right)|\tilde{\sigma}_N^{''}(\omega)| \nonumber \\
&\quad +\Bigg(|e_j(e_i^2\cos(\omega)+1)|(N+1)^2 \nonumber \\
&\quad + \frac{|e_ie_n|}{2}(N+1)+\frac{|e_j|}{4}\Bigg)|\tilde{\sigma}_N^{'}(\omega)|+ |e_i^2e_j||\tilde{\sigma}_N^{'''}(\omega)|.
\end{align*}

We have $|e_j(3e_i^2-1)|\leq 1$, $|e_j(e_i^2\cos(\omega)+1)|\leq 1.1$,  $|e_ie_j|\leq \frac{1}{2}$, $|e_i^2e_j| \leq \frac{2}{3\sqrt{3}}$, and thus for $N\geq 2s \geq 12$
\begin{align*}
|X_j^xX_i^xX_i^y\sigma_N \mp \frac{1}{2} X_n^xX_i^y\sigma_N|
&\leq \frac{1.2 \cdot c_{3,s}}{(N+1)^{s-3}\omega^s}.
\end{align*}
For the last inequality one finds
\begin{align}\label{thirdderiv:5}
X_j^xX_i^xX_j^y\sigma_N \mp \frac{1}{2} X_n^xX_j^y\sigma_N&=\left(e_i(3e_j^2-1)\left(\frac{1+\cos(\omega)}{2\sin(\omega)}\right)\mp\frac{e_ne_j}{2} \right)\tilde{\sigma}_N^{''}(\omega) \nonumber \\
&\quad -\Bigg((2e_ie_j^2-e_i)\left(\frac{1+\cos(\omega)}{2\sin(\omega)}\right)^2+e_ie_j^2\left(\frac{1+\cos(\omega)}{2\sin^2(\omega)}\right) \nonumber \\
&\quad + \frac{e_je_n}{2}\left(\frac{1+\cos(\omega)}{2\sin(\omega)}\right)\Bigg)\tilde{\sigma}_N^{'}(\omega)- e_ie_j^2\tilde{\sigma}_N^{'''}(\omega).
\end{align}
Using similar estimations as before, one shows
\begin{equation*}
|X_j^xX_i^xX_j^y\sigma_N \mp \frac{1}{2} X_n^xX_j^y\sigma_N|\leq \frac{1.2 \cdot c_{3,s}}{(N+1)^{s-3}\omega^s}.
\end{equation*} 
\end{proof}

\subsection{Proof of Lemma \ref{lip:est}}
\begin{proof}
Since
\begin{equation*}
X_i^xX_i^x\sigma_N=\frac{1+\cos(\omega)}{2\sin(\omega)}(1-e_i^2)\tilde{\sigma}_N^{'}(\omega)+e_i^2\tilde{\sigma}_N^{''}(\omega),
\end{equation*}
where $\omega=\omega(y^{-1}x)$ again denotes the rotation angle and $e_i=e_i(y^{-1}x)$ denotes the $i$-th component of the rotation axis. We can write 
\begin{align}\label{start:lipest}
\left(X_i^xX_i^x\sigma_N-\tilde{\sigma}_N^{''}(0)\right)&=(1-e_i^2)\left(\frac{1+\cos(\omega)}{2\sin(\omega)}\tilde{\sigma}_N^{'}(\omega)-\tilde{\sigma}_N^{''}(\omega)\right)\\
&\quad+ \left(\tilde{\sigma}_N^{''}(\omega)-\tilde{\sigma}_N^{''}(0)\right).\nonumber
\end{align}
The second term can be estimated using 
\begin{equation}
(1-\cos(k\omega))\leq \frac{k^2\omega^2}{2},  \quad \omega \in [0, \frac{\delta}{N+1}].
\end{equation}
Therefore, we can estimate using Lemma \ref{2deriv:lowbound}
\begin{align*}
\left(\tilde{\sigma}_N^{''}(\omega)-\tilde{\sigma}_N^{''}(0)\right)&=\frac{1}{\|\tilde{g}_{s-1}\|_{1,N}}2\sum_{k=1}^N\tilde{g}_s\left(\frac{k}{2(N+1)}\right)k^2(1-\cos(k\omega)),\\
&\leq \frac{\omega^2}{2}|\tilde{\sigma}_N^{(4)}(0)| \leq \frac{\tilde{d}_s}{2}(N+1)^2\delta^2.
\end{align*}

We show that the first term in \eqref{start:lipest} is less or equal to zero and bounded in absolute value by the second term. Since
\begin{equation*}
\left(\frac{1+\cos(\omega)}{2\sin(\omega)}\tilde{\sigma}_N^{'}(\omega)-\tilde{\sigma}_N^{''}(\omega)\right)=\frac{2}{\|\tilde{g}_{s-1}\|_{1,N}}\sum_{k=1}^N\tilde{g}_s\left(\frac{k}{2(N+1)}\right)k^2\left(\cos(k\omega)-\frac{1+\cos(\omega)}{2k\sin(\omega)}\sin(k\omega)\right),
\end{equation*}
it is sufficient to show that for each $1\leq k \leq N$
\begin{equation}\label{lessthanzero:start}
\left(\cos(k\omega)-\frac{1+\cos(\omega)}{2k\sin(\omega)}\sin(k\omega)\right)\leq 0, \quad \omega \in [0, \frac{\delta}{N+1}].
\end{equation}
First observe that the lefthand side in \eqref{lessthanzero:start} equals zero at $\omega=0$. Now we show that the lefthand side is also monotonically decreasing. The derivative is given by
\begin{equation}\label{lessthanzero:deriv}
-k\sin(k\omega)+\frac{1}{2}\left(\frac{1+\cos(\omega)}{\sin(\omega)}\right)\left(\frac{\sin(k\omega)}{k\sin(\omega)}-\cos(k\omega)\right).
\end{equation}

To proceed, we show that for each $1\leq k\leq N$
\begin{equation}\label{est4}
\left(\frac{1+\cos(\omega)}{\sin(\omega)}\right)\left(\frac{\sin(k\omega)}{k\sin(\omega)}-\cos(k\omega)\right) \leq k\sin(k\omega).
\end{equation}
On the interval $[0,\frac{\delta}{N+1}]$ this is equivalent to
\begin{equation}\label{est3}
k\cos(\omega) - \frac{\cos(k\omega)\sin(\omega)}{\sin(k\omega)}\leq k - \frac{1}{k}.
\end{equation}
The function on the left hand side equals $k -\frac{1}{k}$ for $\omega=0$. To get the desired estimate we show that the function on the lefthand side of \eqref{est3} attains its maximum on the interval $[0, \frac{\delta}{N+1}]$ at $\omega=0$. The derivative of the left hand side of \eqref{est3} is given by
\begin{equation*}
-k\sin(\omega)+\frac{k\sin(\omega)}{\sin^2(k\omega)}-\frac{\cos(k\omega)\cos(\omega)}{\sin(k\omega)}=k \cot(k\omega)\sin(\omega)\left(\frac{\cos(k\omega)}{\sin(k\omega)}-\frac{\cos(\omega)}{k\sin(\omega)}\right).
\end{equation*}
We have $\frac{k\sin(\omega)}{\cos(\omega)}< \frac{\sin(k \omega)}{\cos(k\omega)}$, due to the power series representation of the tangent function, and therefore
\begin{equation*}
\left(\frac{\cos(k\omega)}{\sin(k\omega)}-\frac{\cos(\omega)}{k\sin(\omega)}\right) < 0.
\end{equation*}
This means the function given by the left hand side of \eqref{est3} is strictly monotonic decreasing on the interval $[0, \frac{\delta}{N+1}]$. Thus it attains its maximum at $\omega=0$. Therefore the function in \eqref{lessthanzero:deriv} is strictly negative, which implies that the inequality \eqref{lessthanzero:start} holds.

The first term of \eqref{start:lipest} can be bounded in absolute value by
\begin{equation*}
\left|\frac{1+\cos(\omega)}{2\sin(\omega)}\tilde{\sigma}_N^{'}(\omega)-\tilde{\sigma}_N^{''}(\omega)\right|=\frac{2}{\|\tilde{g}_{s-1}\|_{1,N}}\sum_{k=1}^N\tilde{g}_s\left(\frac{k}{2(N+1)}\right)k^2\left|\cos(k\omega)-\frac{1+\cos(\omega)}{2k\sin(\omega)}\sin(k\omega)\right|.
\end{equation*}
As seen before in \eqref{lessthanzero:start} we already know that
\begin{align*}
\left|\cos(k\omega)-\frac{1+\cos(\omega)}{2k\sin(\omega)}\sin(k\omega)\right|&=\left(\frac{1+\cos(\omega)}{2k\sin(\omega)}\sin(k\omega)-\cos(k\omega)\right),\\
&=\frac{1+\cos(\omega)}{2k\sin(\omega)}\sin(k\omega)-\cos(k\omega).
\end{align*}
Since $\sin(k\omega)\leq k\sin(\omega)$, we see
\begin{equation*}
\left|\cos(k\omega)-\frac{1+\cos(\omega)}{2k\sin(\omega)}\sin(k\omega)\right|\leq 1-\cos(k\omega)\leq \frac{k^2\omega^2}{2},\quad \omega \in [0, \frac{\delta}{N+1}], 
\end{equation*}
which shows
\begin{equation*}
\left|X_i^xX_i^x\sigma_N-\tilde{\sigma}_N^{''}(0)\right|\leq \frac{\tilde{d}_s}{2}(N+1)^2\delta^2.
\end{equation*}
Moreover,
\begin{equation}
\left|\frac{1+\cos(\omega)}{2\sin(\omega)}\tilde{\sigma}_N^{'}(\omega)-\tilde{\sigma}_N^{''}(\omega)\right|\leq \frac{\tilde{d}_s}{2}(N+1)^2\delta^2, \quad \omega \in  [0, \frac{\delta}{N+1}].
\end{equation}
Similarly we have
\begin{align*}
 X_i^xX_n^x\sigma_N\mp \frac{1}{2}X_j^x\sigma_N=e_ie_n\left(\tilde{\sigma}_N^{''}(\omega)-\frac{1+\cos(\omega)}{2\sin(\omega)}\tilde{\sigma}_N^{'}(\omega)\right),
\end{align*}
which yields, since $|e_ie_j| \leq \frac{1}{2}$,
\begin{equation*}
|X_i^xX_n^x\sigma_N\mp X_j^x\sigma_N| \leq \frac{\tilde{d}_s}{4}(N+1)^2\delta^2.
\end{equation*}
For the third mixed derivatives one has in the case $n\neq i$
\begin{align*}
X_i^xX_i^xX_n^y\sigma_N&=(2e_ne_i^2-e_n)\left(\frac{1+\cos(\omega)}{2\sin(\omega)}\right)\left(\tilde{\sigma}_N^{''}(\omega)-\left(\frac{1+\cos(\omega)}{2\sin(\omega)}\right)\tilde{\sigma}_N(\omega)\right)\\
&\quad + e_ne_i^2\left(\frac{1+\cos(\omega)}{2\sin(\omega)}\right)\left(\tilde{\sigma}_N^{''}(\omega)-\frac{\tilde{\sigma}_N(\omega)}{\sin(\omega)}\right)\\
&\quad \pm e_ie_j \left(\left(\frac{1+\cos(\omega)}{2\sin(\omega)}\right)\tilde{\sigma}_N(\omega)-\tilde{\sigma}_N^{''}(\omega)\right)-e_i^2e_n\tilde{\sigma}_N^{'''}\pm \frac{e_n}{4}\tilde{\sigma}_N'(\omega),
\end{align*}
as seen already in \eqref{thirdderiv:1}. Since $\frac{1+\cos(\omega)}{2}\leq 1$, we have
\begin{align}\label{sin:est}
\left(\frac{1+\cos(\omega)}{\sin(\omega)}\right)\left(\left(\frac{1+\cos(\omega)}{2\sin(\omega)}\right)\frac{\sin(k\omega)}{k}-\cos(k\omega)\right)&\leq\left(\frac{1+\cos(\omega)}{\sin(\omega)}\right)\left(\frac{\sin(k\omega)}{k\sin(\omega)}-\cos(k\omega)\right),\\
& \leq k\sin(k\omega),\nonumber
\end{align}
see \eqref{est4}. Therefore, since $|\frac{e_n}{2}|(|2e_i^2-1|+3e_i^2) \leq 1$,
\begin{align*}
|X_i^xX_i^xX_n^y\sigma_N|&\leq \tilde{d}_s\left((N+1)^3\delta+\frac{1}{4}(N+1)^2\delta^2\right)+\frac{\tilde{c}_s}{4}(N+1)\delta.
\end{align*}
In the case $n=i$ we have, see \eqref{thirdderiv:2},
\begin{align*}
X_i^xX_i^xX_i^y\sigma_N&=2e_i(1-e_i^2) \left(\frac{1+\cos(\omega)}{2\sin(\omega)}\right)\left(\left(\frac{1+\cos(\omega)}{2\sin(\omega)}\right)\tilde{\sigma}_N(\omega)-\tilde{\sigma}_N^{''}(\omega)\right)\\
&\quad+e_i(1-e_i^2)\left(\frac{1+\cos(\omega)}{2\sin(\omega)}\right)\left(\frac{\tilde{\sigma}_N(\omega)}{\sin(\omega)}-\tilde{\sigma}_N^{''}(\omega)\right)-e_i^3\tilde{\sigma}_N^{'''}(\omega).
\end{align*}
Using $\frac{3}{2}|e_i|(1-e_i^2)+|e_i^3|\leq 1$, this results in
\begin{equation*}
|X_i^xX_i^xX_i^y\sigma_N|\leq\tilde{d}_s(N+1)^3\delta.
\end{equation*}

Observe that we have for $n \neq j,i$ the following sign combination
\begin{align*}
X_i^xe_n&=-e_ie_n\left(\frac{1+\cos(\omega)}{2\sin(\omega)}\right)\pm \frac{e_j}{2},\\
X_j^xe_n&=-e_je_n\left(\frac{1+\cos(\omega)}{2\sin(\omega)}\right)\mp \frac{e_i}{2},\\
X_j^xe_i&=-e_je_i\left(\frac{1+\cos(\omega)}{2\sin(\omega)}\right)\pm \frac{e_n}{2}.
\end{align*}
and therefore, as seen in \eqref{thirdderiv:3}, 
\begin{align*}
X_j^xX_i^xX_n^y\sigma_N \mp \frac{1}{2} X_n^xX_n^y\sigma_N&=2e_ie_je_n\left(\frac{1+\cos(\omega)}{2\sin(\omega)}\right)\left(\tilde{\sigma}_N^{''}(\omega)-\left(\frac{1+\cos(\omega)}{2\sin(\omega)}\right)\tilde{\sigma}_N'(\omega)\right)\\
&\quad +e_ie_je_n\left(\frac{1+\cos(\omega)}{2\sin(\omega)}\right)\left(\tilde{\sigma}_N^{''}(\omega)-\frac{\tilde{\sigma}_N'(\omega)}{\sin(\omega)}\right)-e_ie_ne_j\tilde{\sigma}_N^{'''}(\omega)\\
&\quad \pm \frac{e_i^2-e_j^2}{2}\left(\tilde{\sigma}_N^{''}(\omega)-\left(\frac{1+\cos(\omega)}{2\sin(\omega)}\right)\tilde{\sigma}_N'(\omega)\right),
\end{align*}
so we can estimate using \eqref{sin:est} together with $|e_ie_ne_j|\leq \left(\frac{1}{\sqrt{3}}\right)^3\leq \frac{1}{5}$
\begin{equation*}
\left|X_j^xX_i^xX_n^y\sigma_N\mp \frac{1}{2} X_n^xX_n^y\sigma_N\right|\leq\tilde{d}_s\left(\frac{1}{2}(N+1)^3\delta+\frac{1}{4}(N+1)^2\delta^2\right).
\end{equation*}
Similarly, we have, see \eqref{thirdderiv:4} and \eqref{thirdderiv:5},
\begin{align*}
X_j^xX_i^xX_i^y\sigma_N\mp \frac{1}{2} X_n^xX_i^y\sigma_N&=2e_i^2e_j\left(\frac{1+\cos(\omega)}{2\sin(\omega)}\right)\left(\tilde{\sigma}_N^{''}(\omega)-\left(\frac{1+\cos(\omega)}{2\sin(\omega)}\right)\tilde{\sigma}_N'(\omega)\right)\\
& \quad -e_j(1-e_i^2)\left(\frac{1+\cos(\omega)}{2\sin(\omega)}\right)\left(\tilde{\sigma}_N^{''}(\omega)-\frac{\tilde{\sigma}_N'(\omega)}{\sin(\omega)}\right)- e_i^2e_j \tilde{\sigma}_N^{'''}(\omega)\\
&\quad \mp \frac{e_ie_n}{2}\left(\tilde{\sigma}_N^{''}(\omega)-\left(\frac{1+\cos(\omega)}{2\sin(\omega)}\right)\tilde{\sigma}_N'(\omega)\right)\mp\frac{\e_j}{4}\tilde{\sigma}_N^{'}(\omega)
\end{align*}
and
\begin{align*}
X_j^xX_i^xX_j^y\sigma_N \mp \frac{1}{2} X_n^xX_j^y\sigma_N&=
(2e_ie_j^2-e_i)\left(\frac{1+\cos(\omega)}{2\sin(\omega)}\right)\left(\tilde{\sigma}_N^{''}(\omega)-\left(\frac{1+\cos(\omega)}{2\sin(\omega)}\right)\tilde{\sigma}_N(\omega)\right)\\
&\quad + e_ie_j^2\left(\frac{1+\cos(\omega)}{2\sin(\omega)}\right)\left(\tilde{\sigma}_N^{''}(\omega)-\frac{\tilde{\sigma}_N(\omega)}{\sin(\omega)}\right)- e_ie_j^2 \tilde{\sigma}_N^{'''}(\omega)\\
&\quad \pm \frac{e_ne_j}{2}\left(\tilde{\sigma}_N^{''}(\omega)-\left(\frac{1+\cos(\omega)}{2\sin(\omega)}\right)\tilde{\sigma}_N(\omega)\right),
\end{align*}
which yields
\begin{align*}
\left|X_j^xX_i^xX_i^y\sigma_N \mp \frac{1}{2} X_n^xX_i^y\sigma_N\right|&\leq \tilde{d}_s\left((N+1)^3\delta+\frac{1}{8}(N+1)^2\delta^2\right)+\frac{c_s}{4}(N+1)\delta,\\
\left|X_j^xX_i^xX_j^y\sigma_N \mp \frac{1}{2} X_n^xX_j^y\sigma_N\right|&\leq \tilde{d}_s\left((N+1)^3\delta+\frac{1}{8}(N+1)^2\delta^2\right).
\end{align*}
\end{proof}

\bibliography{Exact_Recovery_of_Discrete_Measures_from_Wigner_D_Moments}{}

\begin{thebibliography}{10}

\bibitem{Boelcskei:2015}
C.~Aubel, D.~Stolz, and H.~B\"olcskei.
\newblock A {Theory} of {Super-Resolution} from {Short-Time} {Fourier}
  {Transform} {Measurements}.
\newblock arXiv:1509.01047, 2015.

\bibitem{Bajaj:2013}
C.~Bajaj, B.~Bauer, R.~Bettadupura, and A.~Vollrath.
\newblock {Nonuniform} {Fourier} transforms for rigid-body and multidimensional
  rotational correlations.
\newblock {\em SIAM J. Sci. Comput.}, 35(4), 2013.

\bibitem{Bendory:2014}
T.~Bendory, S.~Dekel, and A.~Feuer.
\newblock Exact recovery of non-uniform splines from the projection onto spaces
  of algebraic polynomials.
\newblock {\em J. Approx. Theory}, 182:7--17, 2014.

\bibitem{Bendory:2015}
T.~Bendory, S.~Dekel, and A.~Feuer.
\newblock {Exact} {Recovery} of {Dirac} {Ensembles} from {Projection} onto
  {Spaces} of {Spherical} {Harmonics}.
\newblock {\em Constr. Approx.}, 42(2):183--207, 2015.

\bibitem{Bendory:2015a}
T.~Bendory, S.~Dekel, and A.~Feuer.
\newblock {Super-Resolution} on the {Sphere} {Using} {Convex} {Optimization}.
\newblock {\em IEEE Transactions on Signal Processing}, 63(9):2253--2262, 2015.

\bibitem{Benedetto:2016}
J.J. Benedetto and W.~Li.
\newblock {Super-resolution} by means of {Beurling} minimal extrapolation.
\newblock arXiv:1601.05761, 2016.

\bibitem{Bojanov:1993}
B.D. Bojanov, H.A. Hakopian, and A.A. Sahakian.
\newblock {\em {Spline} {Functions} and {Multivariate} {Interpolations}}.
\newblock Kluwer Academic Publishers, 1993.

\bibitem{Bunge:1982}
H.J. Bunge.
\newblock {\em {Texture Analysis} in {Material} {Science}}.
\newblock Butterworths, 1982.

\bibitem{Candes:2013}
E.J. Cand\'ez and C.~Fernandez-Granda.
\newblock {Super-Resolution} from {Noisy} {Data}.
\newblock {\em J. Fourier Anal. Appl.}, 19:1229---1254, 2013.

\bibitem{Candes:2014}
E.J. Cand\'ez and C.~Fernandez-Granda.
\newblock {Towards} a {Mathematical} {Theory} of {Super-resolution}.
\newblock {\em Comm. on Pure and Appl. Math.}, 67(6):906--956, 2014.

\bibitem{Candas:2005}
J.E. Castrillon-Candas, V.~Siddavanahalli, and C.~Bajaj.
\newblock {Nonequispaced} {Fourier} transforms for protein-protein docking.
\newblock ICES Report 05-44, Univ. Texas, 2005.

\bibitem{EngAppl:2000}
G.S. Chirikjian and A.B. Kyatkin.
\newblock {\em {Engineering} {Applications} of {Noncummutative} {Harmonic}
  {Analysis}}.
\newblock CRC Press, 2000.

\bibitem{Castro:2012}
Y.~de~Castro and F.~Gamboa.
\newblock {Exact} {Reconstruction} using {Beurling} minimal extrapolation.
\newblock {\em J. Math. Anal. Appl.}, 395(1):336--354, 2012.

\bibitem{deCastro:2015}
Y.~de~Castro, F.~Gamboa, D.~Henrion, and J.-B. Lasserre.
\newblock {Exact} solutions to {Super} {Resolution} on semi-algebraic domains
  in higher dimensions.
\newblock arXiv:1502.02436, 2015.

\bibitem{Kunis:2008}
M.~Gr\"af and S.~Kunis.
\newblock {Stability} {Results} for {Scattered} {Data} {Interpolation} on the
  {Rotation} {Group}.
\newblock {\em Elect. Trans. on Num. Ana.}, 31:30--39, 2008.

\bibitem{Hielscher:2008}
R.~Hielscher, D.~Potts, J.~Prestin, H.~Schaeben, and M.~Schmalz.
\newblock The {Radon} transform on {$SO(3)$}: a {Fourier} slice theorem and
  numerical inversion.
\newblock {\em Inverse Problems}, 24, 2008.

\bibitem{Hielscher:2010}
R.~Hielscher, J.~Prestin, and A.~Vollrath.
\newblock {Fast} summation of {Functions} on the {Rotation} {Group}.
\newblock {\em Math. Geosci.}, 42:773--794, 2010.

\bibitem{Keiner:2012}
J.~Keiner and A.~Vollrath.
\newblock {A} {New} {Algorithm} for the {Nonequispaced} {Fast} {Fourier}
  {Transform} on the {Rotation} {Group}.
\newblock {\em SIAM J. Sci. Comput.}, 34(5):A2599--A2624, 2012.

\bibitem{Kostelec:2008}
P.J. Kostelec and D.N. Rockmore.
\newblock {FFTs} on the {Rotation} {Group}.
\newblock {\em J. Fourier Anal. Appl.}, 14:145--179, 2008.

\bibitem{Kovacs:2003}
J.A. Kovacs, P.~Cha\'on, Y.~Cong, E.~Metwally, and W.~Wriggers.
\newblock Fast rotational matching of rigid bodies by fast {Fourier} transfrom
  acceleration of five degrees of freedom.
\newblock {\em Acta Crystallographica}, Sect. D(59):1371--1376, 2003.

\bibitem{vOhe:2016a}
S.~Kunis, T.~Peter, H.M. M\"oller, and U.~von~der Ohe.
\newblock {Prony's} method on the sphere.
\newblock arXiv:1603.02020.

\bibitem{vOhe:2016}
S.~Kunis, T.~Peter, T.~R\"omer, and U.~von~der Ohe.
\newblock A multivariate generalization of {Prony's} method.
\newblock {\em Linear Algebra and its Applications}, 490:31--47, 2016.

\bibitem{Potts:2007}
S.~Kunis and D.~Potts.
\newblock {Stability} {Results} for {Scattered} {Data} {Interpolation} by
  {Trigonometric} {Polynomials}.
\newblock {\em SIAM J. Scientific Comp.}, 29(4):1403--1419, 2007.

\bibitem{Mhaskar:2000}
H.N. Mhaskar and J.~Prestin.
\newblock {On} the detection of singularities of a periodic function.
\newblock {\em Adv. in Comp. Math.}, 12:95--131, 2000.

\bibitem{Potts:2009}
D.~Potts, J.~Prestin, and A.~Vollrath.
\newblock A fast algorithm for nonequispaced {Fourier} transforms on the
  rotation group.
\newblock {\em Numerical Algorithms}, 52(355), 2009.

\bibitem{Schaeben:2003}
H.~Schaeben and K.G. v.d. Boogart.
\newblock {Spherical} harmonics in texture analysis.
\newblock {\em Tectophysics}, 370:253--268, 2003.

\bibitem{Schmid:2009}
D.~Schmid.
\newblock {\em {Scattered} {Data} {Approximation} on the {Rotation} {Group} and
  {Generalizations}}.
\newblock PhD thesis, TU M\"unchen, 2009.

\bibitem{Schumaker:2007}
L.L. Schumaker.
\newblock {\em {Spline} {Functions}: {Basic} {Theory}}.
\newblock Cambridge University Press, 2007.

\bibitem{Stevensson:2011}
B.~Stevensson and M.~Ed\'en.
\newblock {Interpolation} by fast {Wigner} transforms for rapid calculations of
  magnetic spectra from powders.
\newblock {\em J. Chem. Phys.}, 134, 2011.

\bibitem{Hielscher:2007}
K.G. v.d. Boogart, R.~Hielscher, J.~Prestin, and H.~Schaeben.
\newblock Kernel-based methods for inversion of the radon transform on
  {$SO(3)$} and their application to texture analysis.
\newblock {\em J. Comp. Appl. Math.}, 199:122--140, 2007.

\end{thebibliography}
\bibliographystyle{plain}

\end{document}